\documentclass[11pt]{amsart}
\usepackage{amssymb}
\usepackage{amsmath}
\usepackage{mathtools} %for \mathclap
\usepackage{graphicx,psfrag}
\usepackage{fancyhdr}
\usepackage[top=1.3in, bottom=1.1in, left=1.2in, right=1.2in]{geometry}
\usepackage{array}

\usepackage{enumerate}
\usepackage{paralist}

\numberwithin{equation}{section}

\title{On the exact decomposition threshold for even cycles}

\date{\today}
\author{Amelia Taylor}

\newtheorem{firstthm}{Proposition}[section]
\newtheorem{theorem}[firstthm]{Theorem}
\newtheorem{prop}[firstthm]{Proposition}
\newtheorem{lemma}[firstthm]{Lemma}

\newtheorem{fact}[firstthm]{Fact}

\newdimen\margin   % needed for macros \textdisplay & \ltextdisplay
\def\textno#1&#2\par{%
   \margin=\hsize
   \advance\margin by -4\parindent
          \setbox1=\hbox{\sl#1}%
   \ifdim\wd1 < \margin
      $$\box1\eqno#2$$%
   \else
      \bigbreak
      \hbox to \hsize{\indent$\vcenter{\advance\hsize by -3\parindent
      \it\noindent#1}\hfil#2$}%
      \bigbreak
   \fi}

\newcommand{\claim}[2]{\medskip \noindent \textbf{#1: } \emph{#2}}

\begin{document}

\def\COMMENT#1{}
\def\TASK#1{}

\def\eps{{\varepsilon}}
\newcommand{\ex}{\mathbb{E}}
\newcommand{\pr}{\mathbb{P}}
\newcommand{\cA}{\mathcal{A}}
\newcommand{\cB}{\mathcal{B}}
\newcommand{\cC}{\mathcal{C}}
\newcommand{\cD}{\mathcal{D}}
\newcommand{\cE}{\mathcal{E}}
\newcommand{\cF}{\mathcal{F}}
\newcommand{\cH}{\mathcal{H}}
\newcommand{\cS}{\mathcal{S}}
\newcommand{\cP}{\mathcal{P}}
\newcommand{\N}{\mathbb{N}}
\newcommand{\cupdot}{\mathbin{\mathaccent\cdot\cup}}
\newcommand{\dbip}{\delta_{\text{\textup{bip}}}}
\newcommand{\con}{\text{con}}

\begin{abstract}\noindent
A graph $G$ has a $C_k$-decomposition if its edge set can be partitioned into cycles of length $k$. We show that if $\delta(G)\geq 2|G|/3-1$, then $G$ has a $C_4$-decomposition, and if $\delta(G)\geq |G|/2$, then $G$ has a $C_{2k}$-decomposition, where $k\in \N$ and $k\geq 4$ (we assume $G$ is large and satisfies necessary divisibility conditions). These minimum degree bounds are best possible and provide exact versions of asymptotic results obtained by Barber, K\"uhn, Lo and Osthus. In the process, we obtain asymptotic versions of these results when $G$ is bipartite or satisfies certain expansion properties.
\end{abstract}

\maketitle

\section{Introduction}\label{sec:even:intro}

Let $F$ and $G$ be graphs. We say that $G$ has an \emph{$F$-decomposition} (or is \emph{$F$-decomposable}) if its edge set can be partitioned into copies of $F$. One of the first results in the study of graph decompositions was due to Kirkman~\cite{kirk} who gave conditions for a clique to have a $K_3$-decomposition. His result was generalised by Wilson~\cite{Wilson} who determined when large cliques have $F$-decompositions for arbitrary $F$. When $G$ is not a clique, the problem becomes more challenging and the corresponding decision problem is NP-complete \cite{DT}. 

Clearly, every graph which has an $F$-decomposition must satisfy certain vertex degree and edge divisibility conditions. There have been many recent developments bounding the $F$-decomposition threshold, that is, the minimum degree which ensures an $F$-decomposition in any large graph satisfying the necessary divisibility conditions.
General results on the $F$-decomposition threshold establishing a close connection to its fractional counterpart are obtained in \cite{mindeg} and \cite{stef}. Moreover, \cite{mindeg} determines the asymptotic decomposition threshold for even cycles and \cite{stef} generalises this to arbitrary bipartite graphs. The results in \cite{mindeg} and \cite{stef} can be combined with bounds for the fractional version of this problem in \cite{bklmo} and \cite{dross} to obtain good explicit bounds on the $F$-decomposition threshold. Corresponding results for the multipartite setting (with applications to the completion of partially filled Latin squares) were considered in \cite{rpart}, \cite{Dukes2} and \cite{mont}.
The only known exact minimum degree bound (prior to Theorem~\ref{thm:mainc2k}) was obtained by Yuster~\cite{yusttree} who studied the case when $F$ is a tree.
%Here the minimum degree bound is $\lfloor |G|/2\rfloor$.

From here on, we restrict our attention to the case when $F$ is a cycle.
We say that $G$ is \emph{$C_{k}$-divisible} if $e(G)$ is divisible by $k$ and every vertex of $G$ has even degree. Note that any graph which has a $C_k$-decomposition is necessarily $C_k$-divisible.
For each $k\in \N$ with $k\geq 2$, let us define
	\begin{equation*}
	\delta_k:=
	\begin{cases}
	2/3 & \text{if}\ k=2,\\
	1/2 & \text{if}\ k\geq 3.
	\end{cases}
	\end{equation*}
Barber, K\"uhn, Lo and Osthus \cite{mindeg} proved asymptotically best possible minimum degree bounds for a graph to have a $C_{2k}$-decomposition.

\begin{theorem}[\cite{mindeg}]\label{thm:witheps}
Let $k\in \N$ with $k\geq 2$. For each $\eps>0$, there is an $n_0$ such that every $C_{2k}$-divisible graph $G$ on $n\geq n_0$ vertices with $\delta(G)\geq (\delta_k+\eps)n$ has a $C_{2k}$-decomposition.
\end{theorem}

In this paper we remove the linear error term from Theorem~\ref{thm:witheps} to obtain best possible minimum degree bounds for cycles of all even lengths except length six. We structure the proof into extremal cases where we construct the decompositions directly and non-extremal cases where the iterative absorption approach of \cite{mindeg} and \cite{stef} remains effective. In Proposition~\ref{prop:examples:c2k}, we give constructions which show that our bounds are best possible.

\begin{theorem}\label{thm:mainc2k}
Let $k\in \N$ with $k=2$ or $k\geq 4$. There is an $n_0$ such that every $C_{2k}$-divisible graph $G$ on $n\geq n_0$ vertices with 
\begin{equation*}
\delta(G)\geq \begin{cases}
	2n/3-1 & \text{if}\ k=2,\\
	n/2 & \text{if}\ k\geq 4
	\end{cases}
\end{equation*}
has a $C_{2k}$-decomposition.
\end{theorem}
It is an open problem to determine the exact minimum degree guaranteeing a $C_6$-decomposition, this is discussed in more detail in Section~\ref{sec:conclud}.

Along the way to proving Theorem~\ref{thm:mainc2k}, we also obtain a bipartite version of Theorem~\ref{thm:witheps} which is stated as Theorem~\ref{thm:bipartiteversion} below.
If $G$ is a bipartite graph with vertex classes $A$ and $B$, we introduce the following variant on the minimum degree. Given $0\leq \delta\leq 1$, we will write $\dbip(G)\geq \delta$ if, for each $v\in A$, $d_G(v)\geq \delta|B|$ and for each $v\in B$, $d_G(v)\geq \delta|A|$. This definition is convenient when the bipartite graph is not balanced. Cavenagh~\cite{cav} already studied $C_4$-decompositions and proved a bound of $\dbip(G)\geq 95/96$ ensures a $C_4$-decomposition. Theorem~\ref{thm:bipartiteversion} is asymptotically best possible, see Proposition~\ref{prop:examples:c2k:bip}.

\begin{theorem}\label{thm:bipartiteversion}
Let $k\in \N$ with $k\geq 2$. For each $\eps>0$, there is an $n_0$ such that every $C_{2k}$-divisible bipartite graph $G=(A,B)$ with $n_0\leq |A|\leq |B|\leq 2|A|$%
\COMMENT{the ``$2$" could be replaced by any number but it's good enough here.}
and $\dbip(G)\geq \delta_k+\eps$ has a $C_{2k}$-decomposition.
\end{theorem}

\subsection{Extremal graphs}
In this section we provide extremal constructions which show that Theorem~\ref{thm:mainc2k} is best possible and Theorem~\ref{thm:bipartiteversion} is asymptotically so.

\begin{prop}\label{prop:examples:c2k}
\begin{enumerate}[\rm(i)]
\item There are infinitely many $C_4$-divisible graphs $G$ with $\delta(G)\geq 2|G|/3-2$ and no $C_4$-decomposition.\label{extremalc4}%
\COMMENT{General $K_{s,s}$ construction in \cite{mindeg}.}
\item Let $k\in \N$, $k\geq 2$. There are infinitely many $C_{2k}$-divisible graphs $G$ with $\delta(G)\geq |G|/2-1$ and no $C_{2k}$-decomposition.\label{extremalc2k}
\end{enumerate}
\end{prop}

\begin{figure}[ht]
\centering
\includegraphics[scale=0.47]{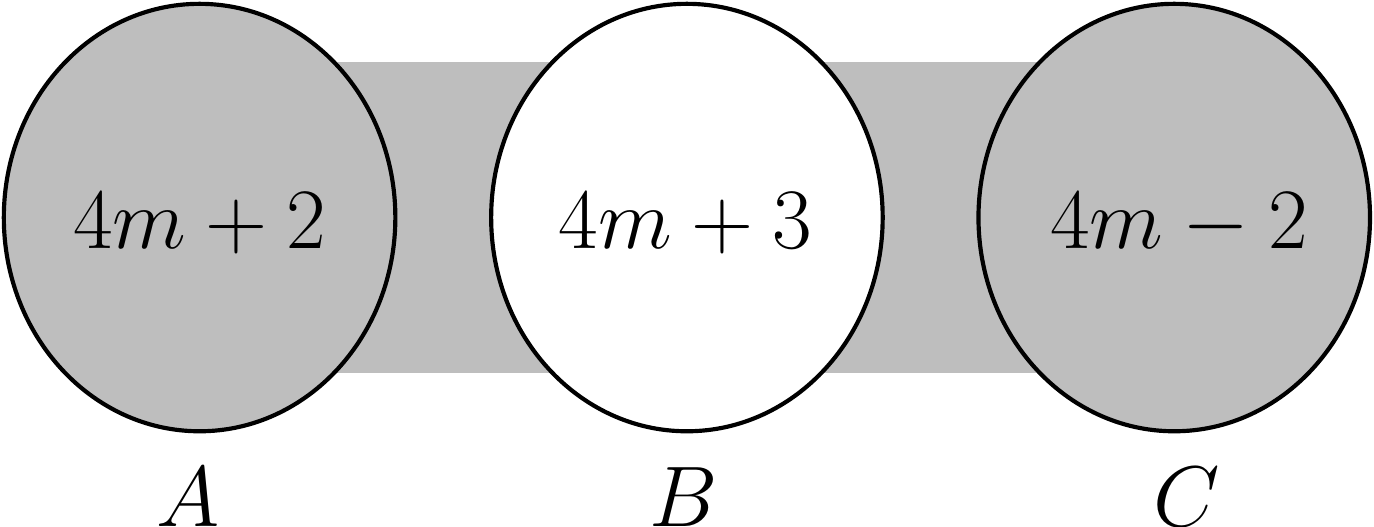}
\caption[The extremal graph for $C_4$, Proposition~\ref{prop:examples:c2k}\eqref{extremalc4}.]{The extremal graph for $C_4$, Proposition~\ref{prop:examples:c2k}\eqref{extremalc4}. All possible edges are present in the shaded regions.}\label{fig:extremalc4}
\end{figure}

\begin{proof}
We begin with \eqref{extremalc4}. Let $m\in \N$ and let $A,B,C$ be disjoint sets of vertices of sizes $4m+2, 4m+3, 4m-2$ respectively. Form a graph $G$ which has vertex set $A\cup B\cup C$. The edge set of $G$ is such that $A$ and $C$ form cliques and $G$ contains all possible edges between $A\cup C$ and $B$.
For each $v\in V(G)$, $d(v)\in \{8m+4, 8m\}$, so every vertex has even degree and $\delta(G)=8m=2|G|/3-2$. We also have
$$e(G)=\binom{4m+2}{2}+8m(4m+3)+\binom{4m-2}{2}=4(12m^2+5m+1).$$
So $G$ is $C_4$-divisible. 
Any copy of $C_4$ in $G$ must use an even number of edges from $G[A]$. But $e(A)=\binom{4m+2}{2}=(2m+1)(4m+1)$ is odd. Hence, $G$ does not have a $C_4$-decomposition.

For \eqref{extremalc2k}, let $n$ be such that $n\equiv 2k+1 \mod{4k}$ and let $G$ be the union of two vertex-disjoint copies of $K_n$. Every vertex in $G$ has degree $n-1=|G|/2-1$ which is even and $2k$ divides $e(G)=n(n-1)$. So $G$ is $C_{2k}$-divisible. But $G$ does not have a $C_{2k}$-decomposition since $2k$ does not divide $\binom{n}{2}$.
\end{proof}

\begin{prop}\label{prop:examples:c2k:bip}
\begin{enumerate}[\rm(i)]
\item There are infinitely many $C_4$-divisible bipartite graphs $G=(A,B)$ with $|A|=|B|$, $\delta(G)\geq 2|A|/3-2$ and no $C_4$-decomposition.\label{extremalc4:bip}
\item Let $k\in \N$, $k\geq 2$. There are infinitely many $C_{2k}$-divisible bipartite graphs $G=(A,B)$ with $|A|=|B|$, $\delta(G)\geq |A|/2 -1$ and no $C_{2k}$-decomposition.\label{extremalc2k:bip}
\end{enumerate}
\end{prop}

\begin{proof}
First, we prove \eqref{extremalc4:bip}. Let $m\in \N$. Start with independent sets $V_1, \dots, V_6$ each of size $2m+1$ and add all edges between $V_i$ and $V_{i+1}$ for each $1\leq i\leq 6$ (consider indices modulo $6$).%
\COMMENT{blow-up of $C_6$}
Remove one copy of $C_6$ between $V_5$ and $V_6$ and let $G$ denote the resulting graph. Then $G$ is bipartite with vertex classes $A:=V_1\cup V_3\cup V_5$ and $B:=V_2\cup V_4\cup V_6$ of size $6m+3$. The degree of each vertex in $G$ is either $4m+2$ or $4m$, both of which are even, and $\delta(G)=4m=2|A|/3-2$. The number of edges in $G$ is $6(2m+1)^2-6=24m(m+1)$. So $G$ is $C_4$-divisible. But $G$ does not have a $C_4$-decomposition. To see this, note that any copy of $C_4$ in $G$ must use an even number of edges between $V_1$ and $V_2$ but $e_G(V_1, V_2)=(2m+1)^2$ is odd.%
\COMMENT{indeed $e_G(V_i, V_{i+1})$ odd for all $i$.}

Now we consider \eqref{extremalc2k:bip}. For each $n\in N$, let $K_{n,n}^-$ denote the graph formed by removing a perfect matching from $K_{n,n}$. Suppose first that $k$ is even. Choose $m\in \N$ such that $m\equiv k+1 \mod 2k$. Let $G$ be the vertex-disjoint union of two copies of $K_{m,m}^-$. Then $G$ is a balanced bipartite graph with vertex classes of size $2m$. Each vertex in $G$ has degree $m-1 \equiv k\mod 2k$ which is even and
$$e(G)=2(m-1)m\equiv 2k(k+1)\equiv 0\mod 2k.$$
So $G$ is $C_{2k}$-divisible. But $G$ does not have a $C_{2k}$-decomposition because $$e(K_{m,m}^-)=(m-1)m \equiv k(k+1) \equiv  k\mod 2k.$$

Now we consider $k$ odd. Choose $m\in \N$ such that $4m\equiv k-1 \mod 2k$ (i.e., choose $m\equiv (k-1)/4 \mod 2k$ if $k\equiv 1\mod 4$ and $m\equiv (3k-1)/4 \mod 2k$ if $k\equiv 3\mod 4$).
Let $G$ be the vertex-disjoint union of $K_{2m+1, 2m+1}^-$ and $K_{2m, 2m}$, so that $G$ is a balanced bipartite graph with vertex classes of size $4m+1$. Note that each vertex in $G$ has degree $2m$ which is even and, since
$$e(G)=2m(4m+1)\equiv 2mk \equiv 0\mod 2k,$$
$G$ is $C_{2k}$-divisible. However, $2k$ does not divide%
\COMMENT{else $4m\equiv 0\not\equiv k-1 \mod 2k$.}
$$e(K_{2m+1, 2m+1}^-)-e(K_{2m, 2m})=2m,$$
so $K_{2m+1, 2m+1}^-$ and $K_{2m, 2m}$ (and hence also $G$) are not $C_{2k}$-decomposable.
\end{proof}

\subsection{Outline of the proof}

Our argument is based on an iterative absorption approach. This method was introduced in \cite{kko} and further developed in the context of $F$-decompositions in \cite{mindeg} and \cite{stef}. In our setting, the idea of iterative absorption is as follows. Let $U$ be a subset of $V(G)$ of constant size and let $U_0\supseteq U_1\supseteq \dots \supseteq U_\ell$ be a decreasing sequence of sets of vertices with $U_\ell:=U$. We use an iterative argument to cover almost all edges of $G$ by copies of $C_{2k}$. Here, it is to our advantage that $C_{2k}$ is bipartite since we can always greedily find an approximate decomposition of $G$ using the Erd\H{o}s-Stone theorem (this is not true for $F$-decompositions in general).
At the end of the $i^\text{th}$~iteration, we are left with a diminishing subgraph $H_i\subseteq G[U_i]$ until, eventually, all that remains is a small leftover $H\subseteq G[U]$.
But we have prepared for $H$ by removing an ``absorber" at the start of the process, a subgraph $A$ of $G$ with the property that $A\cup H$ has a $C_{2k}$-decomposition. This absorber must be able to deal with all possible leftover graphs in $G[U]$, but this is feasible since $U$ only has constant size. Thus we obtain a $C_{2k}$-decomposition of $G$.
So the proof of Theorem~\ref{thm:witheps} using iterative absorption relies on two parts:
\begin{enumerate}
\item $G$ contains an absorber and\label{hurdle1}
\item we can cover all edges in $G-G[U]$. \label{hurdle2}
\end{enumerate}
When we relax the minimum degree condition on $G$ to prove Theorem~\ref{thm:mainc2k}, one or both of these properties can become considerably more challenging to attain.

When the cycle has length at least eight, we need to show that a minimum degree of $|G|/2$ suffices to find a $C_{2k}$-decomposition. If $G$ satisfies a certain expansion property this guarantees many disjoint paths between any pair of vertices, which enables us to show that \eqref{hurdle1} and \eqref{hurdle2} still hold. If $G$ is not an expander, then $G$ has one of two well-defined extremal structures. Either $G$ resembles a complete bipartite graph or the disjoint union of two cliques. In either case, we can construct $C_{2k}$-decompositions directly. We first deal with any edges or vertices which are unusual in some way to leave behind disjoint graphs or bipartite graphs which have high minimum degree. These can be decomposed using the existing Theorem~\ref{thm:witheps} or the bipartite version, Theorem~\ref{thm:bipartiteversion}, (which is proved in Section~\ref{sec:bipartite}).

Cycles of length four are treated separately since in this case we require a higher minimum degree, namely $\delta(G)\geq 2|G|/3-1$. In fact, this minimum degree is sufficient (with room to spare) for \eqref{hurdle2} and it is only finding an absorber which causes any difficulty. We are able to show that any graph which does not contain an absorber will, as in the previous case, have a well-defined structure and we find a $C_4$-decomposition directly.

This paper is organised as follows. In Section~\ref{sec:even:notat}, we introduce the notation which will be used throughout. We construct absorbers in Section~\ref{sec:absorbers}. We prove Theorem~\ref{thm:mainc2k} for $k=2$ in Section~\ref{sec:c4} and for $k\geq 4$ in Section~\ref{sec:longercycle} (see Table~\ref{table:outline} for a guide).
As mentioned above, these proofs rely on decomposition results when the host graph $G$ is bipartite (see Theorem~\ref{thm:bipartiteversion}) and when $G$ is an expander  (see Theorem~\ref{thm:nuexpander}). These results are proved in Sections~\ref{sec:bipartite}~and~\ref{sec:expanderdecomp} respectively.

\newcolumntype{C}[1]{>{\centering\arraybackslash}p{#1}}
\newcolumntype{R}[1]{>{\raggedleft\arraybackslash}p{#1}}
\def\arraystretch{1.5}

\begin{table}[h]
\centering
	\begin{tabular}{R{2.8cm}|C{2.5cm}|C{4cm}}
		&$\mathbf{C_4}$  & $\mathbf{C_{8+}}$\\
	\hline
	\textbf{non-extremal}& Lemma~\ref{lem:notextrem}&Theorem~\ref{thm:nuexpander}\\
	\textbf{extremal}& Lemma~\ref{lem:t1ort2}&Lemmas~\ref{lem:closeclique}~and~\ref{lem:closebip}\\
	\end{tabular}
	\vspace{10pt}
	\caption{Components in the proof of Theorem~\ref{thm:mainc2k}.}\label{table:outline}
\end{table}

\section{Notation and tools}\label{sec:even:notat}

Let $G$ be a graph and let $\cP=\{U_1, \dots, U_k\}$ be a partition of $V(G)$. We write $G[U_1]$ for the subgraph of $G$ induced by $U_1$ and $G[U_1, U_2]$ for the bipartite subgraph of $G$ induced by the vertex classes $U_1$ and $U_2$. We write $G[\cP]:=G[U_1, \dots, U_k]$ for the $k$-partite subgraph of $G$ induced by the partition $\cP$. We say the partition $\cP$ is \emph{equitable} if its parts differ in size by at most one.

Let $U, V\subseteq V(G)$. We write $E_G(U):=E(G[U])$ and $e_G(U):=e(G[U])$. If $U$ and $V$ are disjoint, we let $E_G(U,V):=E(G[U,V])$ and $e_G(U,V):=e(G[U,V])$. For any $v\in V(G)$, $N_G(v,U):=N_G(v)\cap U$ and $d_G(v,U):=|N_G(v,U)|$.
Let $H$ be a graph. We write $G-H$ for the graph with vertex set $V(G)$ and edge set $E(G)\setminus E(H)$. We write $G\setminus H$ for the subgraph of $G$ induced by the vertex set $V(G)\setminus V(H)$. (Note that, in general, $G-H\neq G\setminus H$.)

Let $F$ and $G$ be graphs and let $\eta>0$. We say that a collection $\cF$ of edge-disjoint copies of $F$ in $G$ is an \emph{$\eta$-approximate $F$-decomposition of $G$} if $e(G-\bigcup \cF)\leq \eta |G|^2$. In this paper, the graph $F$ will always be bipartite, so we can greedily apply the Erd\H{o}s-Stone theorem to find an $\eta$-approximate $F$-decomposition of any large graph $G$.
We say that $G$ is \emph{$2$-divisible} if every vertex in $G$ has even degree.%
\COMMENT{eulerian without the connectedness}

We use hierarchies, for example $1/n\ll a\ll b<1$, where constants are chosen from right to left. The notation $a \ll b$ means that there exists an increasing function $f$ for which the result holds whenever $a \leq f(b)$. In order to simplify the presentation, we will not determine these functions explicitly.

Let $m,n,N\in \N$ with $m,n<N$. The \emph{hypergeometric distribution} with parameters $N$, $n$ and $m$ is the distribution of the random variable $X$ defined as follows. Let $S$ be a random subset of $\{1,2, \dots, N\}$ of size $n$ and let $X:=|S\cap \{1,2,\dots, m\}|$. We will frequently use the following simple form of Hoeffding's inequality.

\begin{lemma}[see {\cite[Remark 2.5 and Theorem 2.10]{JLR}}] \label{lem:chernoff}
Let $X\sim B(n,p)$ or let $X$ have a hypergeometric distribution with parameters $N,n,m$.
Then
$\pr(|X - \ex(X)| \geq t) \leq 2e^{-2t^2/n}$.
\end{lemma}

\section{Absorbers}\label{sec:absorbers}

As described earlier, the main idea in the proof of the non-extremal cases of Theorem~\ref{thm:mainc2k} is to cover as many edges of $G$ as possible with copies of $C_{2k}$ using an iterative approach. Then, as long as only a small number of edges remain, we can ``absorb'' these using a special graph which was reserved at the start of the process.
Let $H$ and $H'$ be vertex-disjoint graphs. The graph $A$ is an \emph{$F$-absorber} for $H$ if both $A$ and $A \cup H$ have $F$-decompositions. 
An \emph{$(H,H')_F$-transformer} is a graph $T$ which is edge-disjoint from $H$ and $H'$ and is such that both $T\cup H$ and $T\cup H'$ have $F$-decompositions.
Note that if $H'$ has an $F$-decomposition, then $T\cup H'$ is an $F$-absorber for $H$. So we can use transformers to build an absorber.

The following fact follows directly from $H$ being Eulerian.

\begin{fact}
Let $H$ be any connected $2$-divisible graph%
\COMMENT{i.e., eulerian}
and let $C$ be a cycle of length $e(H)$. There is a graph homomorphism $\phi$ from $C$ to $H$ that is edge-bijective.%
\COMMENT{$H$ has an Euler tour, follow this tour, numbering the vertices every time you visit them - then you can match up the vertices with the cycle (whose vertices are labelled in order around the cycle). Provides a cleaner replacement for the loop graph (which avoids describing identifying vertices and having to make $H$ regular) - but does mean we need to make $H$ CONNECTED.}
\end{fact}

We will make use of the following graphs.
For any $i,k\in \N$, define $L(i,k)$ to be the graph consisting of $i$ copies of $C_{2k}$ with exactly one common vertex. For any graph $H$, we say that $H^\con$ is a \emph{$C_{2k}$-connector} for $H$ if:
\begin{itemize}
\item $H\cup H^\con$ is connected and
\item $H^\con$ has a $C_{2k}$-decomposition.
\end{itemize}
The following simple procedure finds a $C_{2k}$-connector for $H$. Suppose $H$ is not connected and choose vertices $u$ and $v$ which lie in separate components of $H$. Form a copy of $C_{2k}$ containing these vertices by adding two edge-disjoint paths of length $k$ between $u$ and $v$. If the resulting graph $H_1$ is not connected, repeat this process on $H_1$. Eventually, a connected graph $H'$ is obtained with $|H'|\leq (2k-1)|H|$.%
\COMMENT{$|H'|\leq |H|+(|H|-1)(2k-2)$}
The graph $H^\con:=H'-H$ is a $C_{2k}$-connector for $H$.

\subsection{Absorbers for long cycles}\label{sec:longerabs}

The following simple transformer construction suits our purpose. 
Let $H$ be a connected $2$-divisible graph and let $C=u_1u_2\dots u_h$ be a cycle of length $h:=e(H)$ which is vertex-disjoint from $H$. Let $\phi$ be a graph homomorphism from $C$ to $H$ that is edge-bijective. For each $1\leq i\leq h$, let $P_i$ be a path of length $k$ from $u_i$ to $\phi(u_i)$ and let $Q_i$ be a path of length $k-1$ from $u_{i+1}$ to $\phi(u_i)$ (we consider indices modulo $h$). Suppose that the paths $P_i$, $Q_i$ are internally disjoint and that they are edge-disjoint from $H$ and $C$. Note that for each $1\leq i \leq h$, $u_iu_{i+1}\cup P_i\cup Q_i$ and $\phi(u_iu_{i+1})\cup P_{i+1}\cup Q_i$ form copies of $C_{2k}$. So $T:=\bigcup_{i=1}^h (P_i\cup Q_i)$ is a $(C, H)_{C_{2k}}$-transformer and $|T|=2ke(H)$.

\begin{lemma}\label{lem:abs:c2k}
Let $k\in \N$, $k\geq 4$ and $1/n\ll 1/m'\ll 1/m\ll 1/k$.
Let $G$ be a graph on $n$ vertices and let $U\subseteq V(G)$ with $|U|=m$. Suppose that between any pair of vertices $x,y\in V(G)$ there are at least $m'$ internally disjoint paths of length $k-1$.%
\COMMENT{$\delta(G)\geq m'$ for free.}
Then $G$ contains a $C_{2k}$-divisible subgraph $A^*$ such that $|A^*|\leq 2^{m^2}$ and if $H$ is any $C_{2k}$-divisible graph on $U$ that is edge-disjoint from $A^*$ then $A^*\cup H$ has a $C_{2k}$-decomposition.%
\COMMENT{We need to apply this lemma to $G-G[U_i]$ for some suff. large $i$ - doesn't matter that $H$ isn't a subgraph of this graph.}
\end{lemma}

\begin{proof}
Let $H_1, \dots, H_p$ be an enumeration of all possible $C_{2k}$-divisible graphs on $U$ (note that $p\leq 2^{\binom{m}{2}}$). We will find an absorber for each $H_i$. For each $1\leq i\leq p$, find an edge-disjoint $C_{2k}$-connector $H_i^\con\subseteq G-G[U]$ using the procedure outlined above. Each $H_i':=H_i\cup H_i^\con$ is $C_{2k}$-divisible and $|H_i'|\leq (2k-1)m$.

For each $1\leq i\leq p$, let $h_i:=e(H_i')$, let $C^i$ be a cycle of length $h_i$ and let $J_i$ be a copy of the graph $L(h_i/2k,k)$, defined at the beginning of this section. Find copies of $C^i$ and $J_i$ in $G$ which are vertex-disjoint from each other and from the graphs $H_i'$.%
\COMMENT{easy to find using `many disjoint paths' property}
Find a $(H_i',C^i)_{C_{2k}}$-transformer $T_i$ and a $(C^i,J_i)_{C_{2k}}$-transformer $T_i'$ in $G$ (such that $T_i$ and $T_i'$ are edge-disjoint and avoid all edges fixed so far). It is easy to find these transformers using the construction described above since $G$ contains many internally disjoint paths of length $k-1$ (and hence $k$ also) between any pair of vertices. Then $T_i\cup C^i\cup T_i'\cup J_i$ is an absorber for $H_i'$. Hence $A_i:=H_i^\con\cup T_i\cup C^i\cup T_i'\cup J_i$ is an absorber for $H_i$.  Letting $A^*:=\bigcup_{i=1}^p A_i$ and noting $|A^*|\leq 4kh_ip\leq 2^{m^2}$ completes the proof.%
\COMMENT{$|A^*|\leq 4kh_ip\leq 2k(2km)^22^{\binom{m}{2}}\leq 2^{m^2}$}
\end{proof}

\subsection{$C_4$-absorbers}\label{sec:specialc4trans}

For cycles of length four, we will require the following alternative construction of a transformer.%
\COMMENT{simple construction no good as we need edges between $H$ and $C$.}
This is exactly the construction given in \cite{mindeg} and it is illustrated in Figure~\ref{fig:c4transformer}.

\begin{figure}[ht]
\centering
\includegraphics[scale=0.7]{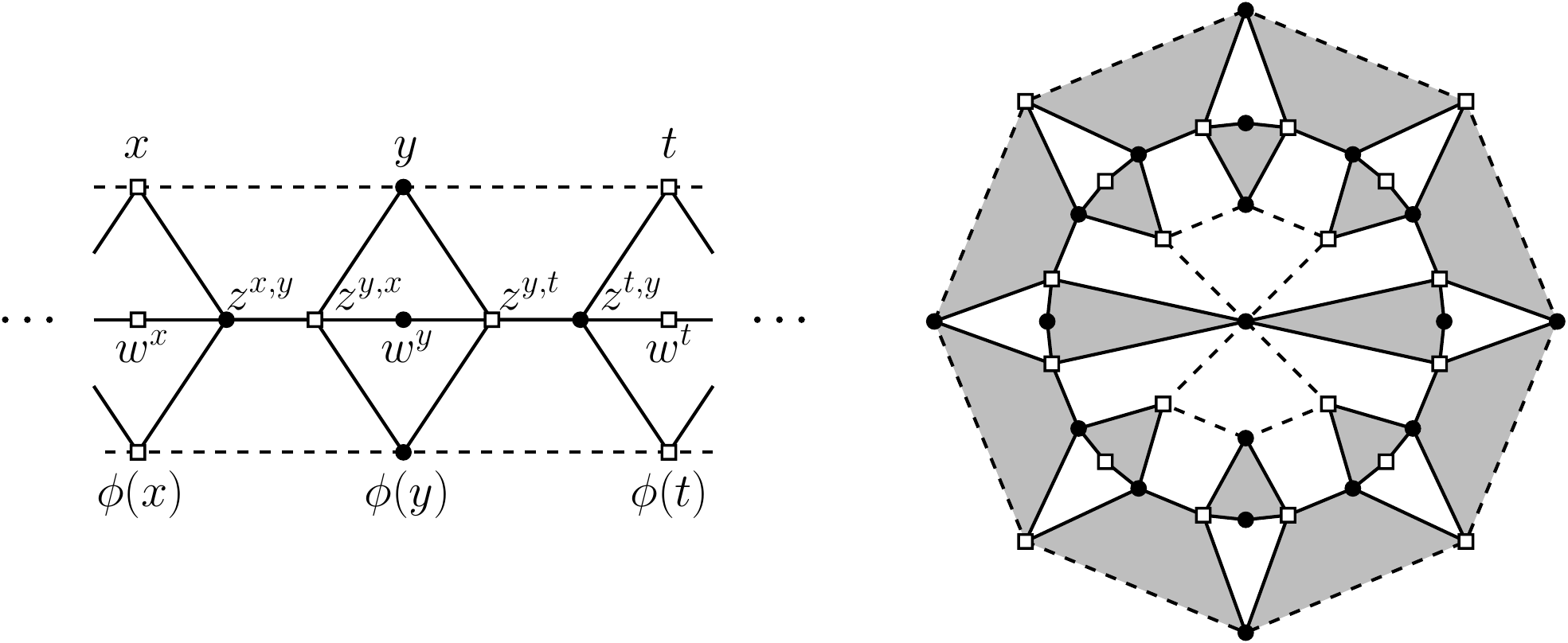}
\caption[The transformer construction for cycles of length four.]{The transformer construction for cycles of length four (left) and a $(C_8, L(2,2))_{C_4}$-transformer (right). The square/round vertices give a bipartition of the transformer which is used by Lemma~\ref{lem:abs:bip}.}\label{fig:c4transformer}
\end{figure}

Let $H$ be a connected, $C_4$-divisible graph and let $C$ be a cycle of length $e(H)$. Suppose $H$ and $C$ are vertex-disjoint. Let $\phi$ be an edge-bijective graph homomorphism from $C$ to $H$. For each $xy\in E(C)$, choose a set of vertices $Z^{xy}:=\{z^{x,y}, z^{y,x}\}$ and, for each $x\in V(C)$, choose a vertex $w^x$. Choose the vertices so that $V(H)$, $V(C)$, $Z^e$, $Z^{e'}$, $\{w^x\}$ and $\{w^{x'}\}$ are disjoint for all distinct $e, e'\in E(C)$ and all distinct $x, x'\in V(C)$. Let
\begin{itemize}
	\item $E_1:=\{xz^{x,y}, yz^{y,x}: xy\in E(C)\}$;
	\item $E_2:=\{z^{x,y}z^{y,x}: xy\in E(C)\}$;
	\item $E_3:=\{\phi(x)z^{x,y}, \phi(y)z^{y,x}: xy\in E(C)\}$;
	\item $E_4:=\{w^xz^{x,y}: xy\in E(C)\}$.
\end{itemize}
The transformer $T$ has $V(T):=V(H) \cup V(C)\cup \bigcup_{e\in E(C)}Z^e \cup \bigcup_{x\in V(C)} \{w^x\}$ and $E(T):=\bigcup_{i=1}^4E_i$. Note that $|T|\leq 5|C|= 5e(H)$. 
To see that $T$ is a $(C, H)_{C_4}$-transformer, it remains to verify that both $C\cup T$ and $H\cup T$ have $C_4$-decompositions (the details are given in Section~8 of \cite{mindeg}).

\subsection{Finding absorbers in a bipartite setting}

We must also be able to find absorbers when the host graph $G$ is bipartite.

\begin{lemma}\label{lem:abs:bip}
Let $k\in \N$, $k\geq 2$ and $1/n\ll 1/m'\ll 1/m\ll 1/k$.
Let $G=(A,B)$ be a bipartite graph with $|A|,|B|\geq n$ and let $U\subseteq V(G)$ with $|U|=m$.
Suppose that for each $v\in A$, $d(v)\geq \delta_k |B|+m'$ and, for each $v\in B$, $d(v)\geq \delta_k |A|+m'$.
Then $G$ contains a $C_{2k}$-divisible subgraph $A^*$ such that $|A^*|\leq 2^{m^2}$ and if $H$ is any $C_{2k}$-divisible graph on $U$ that is edge-disjoint from $A^*$ then $A^*\cup H$ has a $C_{2k}$-decomposition.%
\COMMENT{We need to apply this lemma to $G-G[U_i]$ for some suff. large $i$ - doesn't matter that $H$ isn't a subgraph of this graph.}
\end{lemma}

The proof is very similar to that of Lemma~\ref{lem:abs:c2k} so we omit the details and restrict ourselves to the following outline. For $k\geq 3$ we find transformers using the construction given in Section~\ref{sec:longerabs} and for $C_4$ we use the construction described in Section~\ref{sec:specialc4trans}. The following observations allow us to find absorbers:
\begin{itemize}
\item Given a connected, $2$-divisible graph $H$ and a vertex-disjoint cycle $C$ of length $e(H)$ on $(A,B)$, there is a bipartition of the $(C,H)_{C_{2k}}$-transformer which respects the bipartitions of $V(H)$ and $V(C)$ (with a suitable choice of the graph homomorphism $\phi$).%
\COMMENT{e.g. in the $C_4$ construction, want $x$ and $\phi(x)$ in the same vertex class. In the general construction, $x$ and $\phi(x)$ in the same class if and only if $k$ is even. No problem, just shift around the cycle by one if needs be.}
An example for cycles of length four is given in Figure~\ref{fig:c4transformer}.

\item For $k\geq 3$, $(C,H)_{C_{2k}}$-transformers are constructed from a collection of internally-disjoint paths of length $k$ or $k-1$ between vertices in $C$ and $H$. Since $\delta(G)\geq n/2+m'$, any two vertices in $A$ have at least $2m'$ common neighbours in $B$ and vice versa, so we can find the transformers greedily.

\item List the vertices of the $(C,H)_{C_{4}}$-transformer described in Section~\ref{sec:specialc4trans} so that they appear in the following order: $V(C\cup H)$, $\bigcup_{e\in E(C)} Z^e$, $\bigcup_{x\in V(C)} \{w^x\}$. Each vertex in the transformer has at most three of its neighbours appearing before itself in this list. Since $\delta(G)\geq 2n/3+m'$, any three vertices in $A$ have at least $3m'$ common neighbours in $B$ and vice versa. So we can greedily embed the vertices of the transformer in this order.
\end{itemize}

\section{Cycles of length four}\label{sec:c4}

\subsection{Case distinction}

For cycles of length four, the $\eps n$ term in Theorem~\ref{thm:witheps} is required only to find the absorber in the proof. We show that a minimum degree of $2n/3-1$ suffices by observing that any such graph either contains an absorber or has one of two extremal structures pictured in Figure~\ref{fig:types} (both of which have $C_4$-decompositions).

\begin{figure}[ht]
\centering
\includegraphics[scale=0.45]{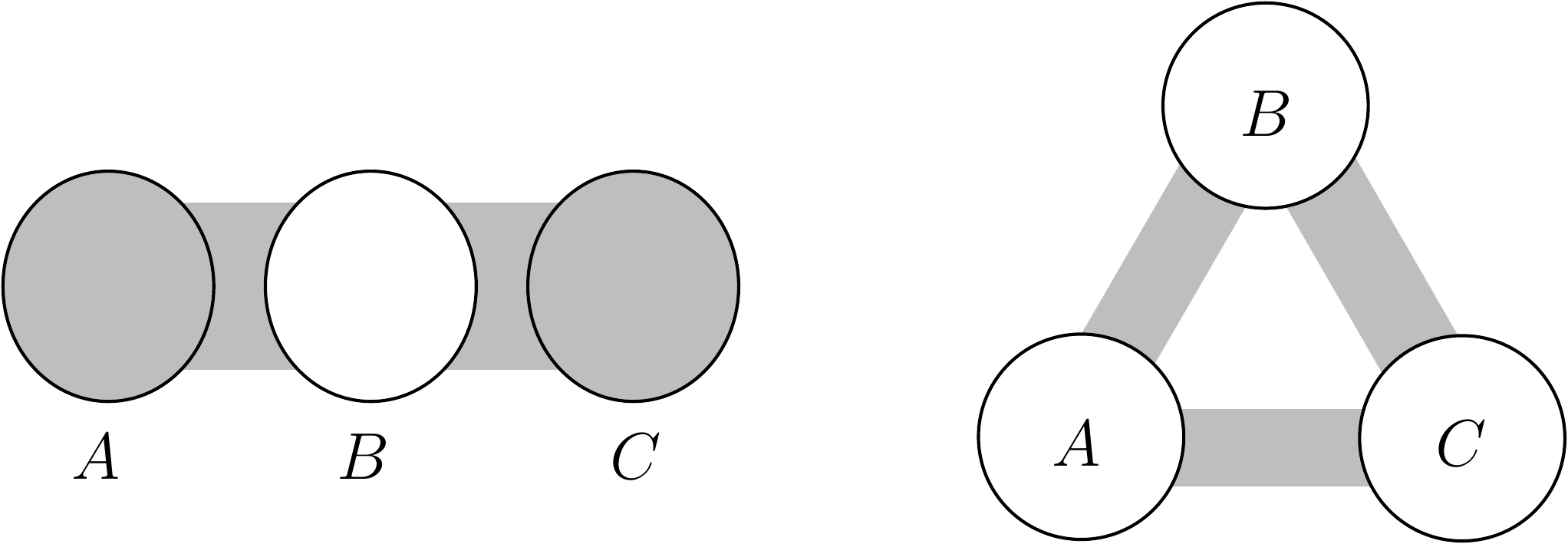}
\caption[Extremal graphs of type 1 and type 2.]{If $G$ is extremal and $\delta(G)\geq 2n/3-1$, then $G$ resembles the graph on the left (type 1) or the right (type 2). Here $|A|, |B|, |C|\sim n/3$ and shaded areas are dense.}\label{fig:types}
\end{figure}

We say that a graph $G$ on $n$ vertices is \emph{$m$-extremal} if there exist disjoint sets $S, T\subseteq V(G)$ such that $|S|, |T|\geq n/3-m$ which satisfy one of the following:
\begin{itemize}
\item $e(S,T)=0$; (\emph{Type 1})
\item $e(S)=e(T)=0$. (\emph{Type 2})
\end{itemize}
If $G$ is not close to being $m$-extremal, Lemma~\ref{lem:notextrem} finds a $C_4$-decomposition.

\begin{lemma}\label{lem:notextrem}
Let $n, m_1, m_2\in \N$ with
$1/n\ll 1/m_1\ll 1/m_2\ll 1$.
Let $G$ be a $C_4$-divisible graph on $n$ vertices with $\delta(G)\geq 2n/3-1$. Suppose that for every spanning%
\COMMENT{so the extremal definition uses the same `$n$'.}
subgraph $G'$ of $G$ such that $\delta(G')\geq 2n/3-m_2$, $G'$ is not $m_1$-extremal. Then $G$ has a $C_4$-decomposition.
\end{lemma}

If Lemma~\ref{lem:notextrem} does not apply, then $G$ has a subgraph $G'$ which is $m_1$-extremal and has $\delta(G')\geq 2n/3-m_2\geq 2n/3-m_1$. In this case, we use the following result.

\begin{lemma}\label{lem:t1ort2}
Let $n,m\in \N$ with
$1/n\ll 1/m\ll 1$.
Let $G$ be a $C_4$-divisible graph on $n$ vertices with $\delta(G)\geq 2n/3-1$. Suppose that there exists a spanning subgraph $G'$ of $G$ such that $\delta(G')\geq 2n/3-m$ and $G'$ is $m$-extremal of
\begin{inparaenum}[\rm(i)]
	\item type 1 or \label{type1}
	\item type 2. \label{type2}
\end{inparaenum}
Then $G$ has a $C_4$-decomposition.
\end{lemma}

So, together, Lemmas~\ref{lem:notextrem}~and~\ref{lem:t1ort2} imply Theorem~\ref{thm:mainc2k} when $k=2$.

\subsection{$G$ is not extremal}

In this section we prove Lemma~\ref{lem:notextrem}, which finds a $C_4$-decomposition of $G$ whenever $G$ is not extremal. 
Let $G$ be a graph on $n$ vertices. A \emph{$(\delta, \mu,m)$-vortex} in $G$ (as defined in \cite{stef}) is a sequence $U_0 \supseteq U_1 \supseteq \dots \supseteq U_\ell$ such that:
\begin{itemize}
\item $U_0=V(G)$;
\item $|U_i|=\lfloor \mu|U_{i-1}| \rfloor$, for all $1\leq i\leq \ell$, and $|U_\ell|=m$;
\item $d_G(x,U_i) \ge \delta |U_i|$, for all $1\leq i\leq \ell$ and all $x \in U_{i-1}$.
\end{itemize}
We use Lemma~4.3 from~\cite{stef} to find a vortex in $G$.

\begin{lemma}[\cite{stef}]\label{lem:vortex}
Let $0\leq \delta \leq 1$. For all $0<\mu<1$, there exists an $m_0=m_0(\mu)$ such that for all $m'\geq m_0$ the following holds. Whenever $G$ is a graph on $n\geq m'$ vertices with $\delta(G)\geq \delta n$, then $G$ has a $(\delta-\mu, \mu, m)$-vortex for some $\lfloor \mu m'\rfloor \leq m \leq m'$.
\end{lemma}

The following result (taken from the more general statement for $F$-decompositions, Lemma~5.1 in~\cite{stef}) finds an approximate $C_4$-decomposition of $G$ leaving only a very small (and very restricted) leftover $H$.

\begin{lemma}[\cite{stef}]\label{lem:covermost}
Let $1/m\ll \mu$. Let $G$ be a $C_4$-divisible graph with $\delta(G)\geq (1/2+3\mu)|G|$ and let $U_0\supseteq U_1\supseteq \dots \supseteq U_\ell$ be a $(1/2+4\mu, \mu, m)$-vortex in $G$. Then there exists $H\subseteq G[U_\ell]$ such that $G-H$ is $C_4$-decomposable.
\end{lemma}

We must prove the following lemma which reserves an absorber that can be used to deal with this leftover graph $H$.

\begin{lemma}\label{lem:abs:c4}
Let $n, m_1, m_2, m_3\in \N$ with
$1/n\ll 1/m_1\ll 1/m_2\ll 1/m_3\ll 1$.
Let $G$ be a graph on $n$ vertices with $\delta(G)\geq 2n/3-m_2$. Suppose that $G$ is not $m_1$-extremal. Let $U\subseteq V(G)$ with $|U|=m_3$. Then $G$ contains a $C_4$-divisible subgraph $A^*$ with $|A^*|\leq 2^{m_3^2}$ such that for any $C_4$-divisible graph $H$ on $U$ that is edge-disjoint from $A^*$, the graph $A^*\cup H$ has a $C_4$-decomposition.%
\COMMENT{We need to apply this lemma to $G-G[U_i]$ for some suff. large $i$ - doesn't matter that $H$ isn't a subgraph of this graph.}
\end{lemma}

Lemma~\ref{lem:notextrem} follows directly from these results.

\begin{proofof}{Lemma~\ref{lem:notextrem}} (Assuming Lemma~\ref{lem:abs:c4}.)
Let $m_3\in \N$ and $\mu$ be such that 
$$1/n\ll 1/m_1\ll 1/m_2\ll 1/m_3\ll \mu\ll 1.$$
Apply Lemma~\ref{lem:vortex} to $G$ to find a $(2/3-2\mu, \mu, m_3)$-vortex%
\COMMENT{$2/3-2\mu$ (not $2/3-\mu$) because our minimum degree is $2n/3-1$.}
$U_0\supseteq U_1\supseteq \dots \supseteq U_\ell$ in $G$. Define $\ell_0:=\lceil \log_{\mu}(m_2/n)\rceil+1$. We have%
\COMMENT{\begin{align*} |U_{\ell_0}|&=\lfloor \mu |U_{\ell_0-1}| \rfloor= \lfloor \mu\lfloor \mu |U_{\ell_0-1}| \rfloor\rfloor\\& =\dots \geq \mu(\mu(\mu(\dots (\mu n-1)-1)-1)-1)-1= \mu^{\ell_0}n-(\mu^{\ell_0-1}+\mu^{\ell_0-2}+\dots+1) \\&\geq \mu^{\ell_0}n-2\geq \mu^2 m_2-2 \end{align*}}
$$\mu^2 m_2-2\leq \mu^{\ell_0}n-2\leq |U^{\ell_0}|\leq \mu^{\ell_0}n \leq \mu m_2.$$

Let $G':=G-G[U_{\ell_0}]$. We have $\delta(G')\geq 2n/3-1-|U_{\ell_0}|\geq 2n/3-m_2$, so $G'$ is not $m_1$-extremal. Apply Lemma~\ref{lem:abs:c4} to the graph $G'$ with $U_\ell$ playing the role of $U$ to find $A^*\subseteq G'$ as in the lemma. We have $\Delta (A^*)\leq |A^*|\leq 2^{m_3^2}\leq |U_{\ell_0}|/10$, so $U_0\supseteq U_1\supseteq \dots \supseteq U_\ell$ is a $(1/2+4\mu, \mu, m_3)$-vortex in $G^*:=G-A^*$. Then apply Lemma~\ref{lem:covermost} to $G^*$ to find $H\subseteq G^*[U_\ell]$ such that $G^*-H$ has a $C_4$-decomposition. Observing that $A^*\cup H$ has a $C_4$-decomposition (by Lemma~\ref{lem:abs:c4}) completes the proof.
\end{proofof}

%%%%FINDING A C4-ABSORBER

To prove Lemma~\ref{lem:abs:c4}, we will find a $C_4$-absorber for each possible $C_4$-divisible graph on $U$. We will use the transformer construction which was given in Section~\ref{sec:specialc4trans} and embed the vertices of the transformer in the order: $V(H\cup C), \bigcup_{e\in E(C)}Z^e, \bigcup_{x\in V(C)} \{w^x\}$.
The difficulty arises when we try to embed the vertices in $\bigcup_{e\in E(C)}Z^e$ since, unlike in \cite{mindeg}, we can no longer guarantee that any set of three vertices will have a common neighbour.
 
We will say that the edge $v_xv_y$ \emph{transforms} $xy$ to $\phi(x)\phi(y)$ if $v_x\in N(x)\cap N(\phi(x))$ and $v_y\in N(y)\cap N(\phi(y))$.
Suppose we are transforming the edge $xy$ to $\phi(x)\phi(y)$. We are able to do this if there is an edge between $N(x)\cap N(\phi(x))$ and $N(y)\cap N(\phi(y))$. 
These ``transforming" edges are related to the vertices in $\bigcup_{e\in E(C)}Z^e$. That is, for each $xy\in E(C)$, the edge $z^{x,y}z^{y,x}$ transforms $xy$ to $\phi(x)\phi(y)$. This suggests that we will be able to find an absorber as long as there do not exist $X, Y\subseteq V(G)$ with $|X|, |Y|\sim n/3$ and $e(X,Y)=0$ (note that $X$ and $Y$ are not necessarily disjoint, unlike in the definition of $m$-extremal).

\begin{proofof}{Lemma~\ref{lem:abs:c4}}
Let $H_1, \dots, H_p$ be an enumeration of all possible $C_4$-divisible graphs on $U$  and note that $p\leq 2^{\binom{m_3}{2}}$. For each $1\leq i\leq p$, find an edge-disjoint $C_4$-connector $H_i^\con\subseteq G-G[U]$ (using the procedure given in Section~\ref{sec:absorbers}). Each $H_i':=H_i\cup H_i^\con$ is $C_4$-divisible and $|H_i'|\leq 3m_3$.
Let $h_i:=e(H_i')$, let $C^i$ be a cycle of length $h_i$ and let $J_i$ be a copy of $L(h_i/4,2)$. Our strategy is as follows.
Suppose that $G\setminus \bigcup_{i=1}^p H_i'$ contains vertex-disjoint copies of $C^i$ and $J_i$ such that we are able to find edge-disjoint $(C^i,H_i')_{C_4}$- and $(C^i,J_i)_{C_4}$-transformers $T_i$ and $T_i'$. Then we can combine these to obtain a $C_4$-absorber $A_i$ for $H_i$ as in the proof of Lemma~\ref{lem:abs:c2k} (more precisely, letting $A_i:=H_i^\con\cup T_i\cup C^i\cup T_i'\cup J_i$). We use the following claim.

\claim{Claim}{There exist vertex-disjoint copies of $C^1, \dots, C^p, J_1, \dots, J_p$ in $G\setminus \bigcup_{i=1}^p H_i'$ such that the following holds. Let $W\subseteq V(G)$ with $|W|\leq m_2$. For any $1\leq i \leq p$, any $xy\in E(C^i)$ and any $\phi(x)\phi(y) \in E(G)$ there is an edge $v_xv_y\in E(G\setminus W)$ which transforms $xy$ to $\phi(x)\phi(y)$.}
\medskip

We consider two cases.

\claim{Case 1}{For all sets $X, Y\subseteq V(G)$ with $|X|, |Y|\geq n/3-3m_2$, $e_G(X,Y)>0$.}

Find vertex-disjoint copies of $C^1, \dots, C^p, J_1, \dots, J_p$ (anywhere) in $G\setminus \bigcup_{i=1}^p H_i'$. Consider any $W\subseteq V(G)$ with $|W|\leq m_2$, any $xy, \phi(x)\phi(y)\in E(G)$.
Let $X:=(N_G(x)\cap N_G(\phi(x)))\setminus W$ and $Y:=(N_G(y)\cap N_G(\phi(y)))\setminus W$. Note that
$$|X|, |Y|\geq 2\delta(G)-n-|W|\geq n/3-3m_2,$$
so $e_G(X,Y)>0$. Any edge $v_xv_y\in E_G(X,Y)$ transforms $xy$ to $\phi(xy)$.

\claim{Case 2}{There exist $X, Y\subseteq V(G)$ with $|X|, |Y|\geq n/3-3m_2$ such that $e_G(X,Y)=0$.}

Since $G$ is not $m_1$-extremal, $X\cap Y\neq \emptyset$. Let $v\in X\cap Y$ and note that $N_G(v)\subseteq V(G)\setminus (X\cup Y)$. So $|X\cup Y|\leq n/3+m_2$ and 
$$|X\cap Y|\geq 2(n/3-3m_2)-(n/3+m_2)=n/3-7m_2.$$
Let $X'\subseteq X\cap Y$ of size $\lfloor n/3\rfloor -7m_2$. Note that $e_G(X')=0$.

Let $m:=m_1/10$. For each $i\in \{m, n/3-\sqrt m\}$, let $U_{i}:=\{v:v\in V(G)\setminus X', d_G(v, X')\leq i\}$. We have
$$|X'|(2n/3-m_2)\leq e_G(X', V(G)\setminus X')\leq |U_{i}|i+(n-|X'|-|U_{i}|)|X'|$$
which yields
$$|U_{i}|\leq \frac{|X'|(n/3-|X'|+m_2)}{|X'|-i}\leq \frac{|X'|(8m_2+1)}{|X'|-i}.$$
Thus, we have $|U_{m}|\leq 9m_2$ and $|U_{n/3-\sqrt m}|\leq n/100$.
Set $X'':=X'\cup U_{m}$, $Y':=V(G)\setminus X''$ and $Y'':=Y'\setminus U_{n/3-\sqrt m}$.
Note that:
\begin{enumerate}[(i)]
	\item for every $v\in X''$, $d_G(v, Y')\geq 2n/3-2m$;\label{item:absc4:1}
	\item for every $v\in Y'$, $d_G(v, X'')\geq m$ and $d_G(v, Y')\geq 2n/3-m_2-|X''|$;\label{item:absc4:2}
	\item for every $v\in Y''$, $d_G(v, X'')\geq n/3-\sqrt m$;\label{item:absc4:3}
	\item $n/3-8m_2\leq |X''|\leq n/3+2m_2$ and $2n/3-2m_2\leq |Y'|\leq 2n/3+8m_2$.\label{item:absc4:4}
\end{enumerate}
Find vertex-disjoint copies of $C^1, \dots, C^p, J_1, \dots, J_p$ in $G\setminus\bigcup_{i=1}^p H_i'$ such that each cycle $C^i\subseteq G[X'', Y'']$. Consider any $W\subseteq V(G)$ with $|W|\leq m_2$, any $1\leq i\leq p$, any $xy\in E(C^i)$ and any $\phi(x)\phi(y)\in E(G)$. We will assume, without loss of generality, that $x\in X''$ and $y \in Y''$.

Suppose first that $\phi(x),\phi(y) \in X''$. Note that \eqref{item:absc4:1} and \eqref{item:absc4:2} imply
\begin{align*}
|N_G(y,Y')\cap N_G(\phi(y),Y')|&\geq (2n/3-m_2-|X''|)+(2n/3-2m)-|Y'|\\
&= n/3-2m-m_2\geq n/3-3m.
\end{align*}
%=2n/3-2m+2n/3-m_2-|X''|-(n-|X''|)
Choose $v_y$ to be any vertex in $(N_G(y,Y')\cap N_G(\phi(y),Y'))\setminus W$. By \eqref{item:absc4:2} and \eqref{item:absc4:4}, $v_y$ has at least $2n/3-m_2-|X''|>n/4$ neighbours in $Y'$. Since
\begin{align*}
|N_G(x, Y')\cap N_G(\phi(x), Y')|\stackrel{\eqref{item:absc4:1}}{\geq} 2(2n/3-2m)-|Y'|\stackrel{\eqref{item:absc4:4}}{\geq} |Y'|-5m,
\end{align*}
$v_y$ has many neighbours in $(N_G(x, Y')\cap N_G(\phi(x), Y'))\setminus W$, choose any one of these for $v_x$.

Now suppose that $\phi(x),\phi(y)\in Y'$. It follows from \eqref{item:absc4:2}--\eqref{item:absc4:4} that
\begin{align*}
|N_G(y,X'')\cap N_G(\phi(y),X'')|\geq (n/3-\sqrt m)+m-(n/3+2m_2)=m-\sqrt m -2m_2\geq m/2.
\end{align*}
Choose any vertex from $(N_G(y,X'')\cap N_G(\phi(y),X''))\setminus W$  for $v_y$. This vertex is adjacent to all but at most $3m$ vertices in $Y'$, by \eqref{item:absc4:1} and \eqref{item:absc4:4}. Use \eqref{item:absc4:1} and \eqref{item:absc4:2} to see that  $|N_G(x, Y')\cap N_G(\phi(x), Y')|\geq n/3-3m$. Thus $v_y$ must have many neighbours in $(N_G(x,Y')\cap N_G(\phi(x), Y'))\setminus W$. Choose any suitable vertex for $v_x$.

A similar argument deals with the case when $\phi(x)\in X''$ and $\phi(y)\in Y'$. We use that
$$|N_G(y,X'')\cap N_G(\phi(y),X'')|\geq m/2\;\text{ and }\;|N_G(x, Y')\cap N_G(\phi(x), Y')|\geq |Y'|-5m$$
to find suitable vertices $v_y\in (N_G(y,X'')\cap N_G(\phi(y),X''))\setminus W$ and $v_x\in (N_G(x, Y')\cap N_G(\phi(x), Y'))\setminus W$.

Finally, suppose that $\phi(x)\in Y'$ and $\phi(y)\in X''$. We again use \eqref{item:absc4:1} and \eqref{item:absc4:2} to see that 
$$|N_G(x, Y')\cap N_G(\phi(x), Y')|, |N_G(y, Y')\cap N_G(\phi(y), Y')|\geq n/3-3m.$$
Let
\begin{align*}
Y_x&:=(N_G(x, Y')\cap N_G(\phi(x), Y'))\setminus W\text{ and}\\
Y_y&:=(N_G(y, Y')\cap N_G(\phi(y), Y'))\setminus W,
\end{align*}
so $|Y_x|,|Y_y|\geq n/3-4m$. If $e_G(Y_x, Y_y)>0$ choose any $v_xv_y\in E_G(Y_x,Y_y)$. Suppose then that $e_G(Y_x, Y_y)=0$. Note that $Y_x\cap Y_y\neq \emptyset$, else $G$ is $m_1$-extremal of type 1. So, as previously, we can let $v\in Y_x\cap Y_y$ and note that $N_G(v)\subseteq V(G)\setminus (Y_x\cup Y_y)$. So $|Y_x\cup Y_y|\leq n/3+m_2$ and $$|Y_x\cap Y_y|\geq 2(n/3-4m)-(n/3+m_2)\geq n/3-9m.$$
But then $G$ is $m_1$-extremal of type 2 (take $S:=X'$ and $T:=Y_x\cap Y_y$) which is a contradiction. This completes the proof of the claim.

\medskip We now explain how to use the claim to find, for each $1\leq i\leq p$, a $(C^i, H_i')_{C_4}$-transformer (and $(C^i,J_i)_{C_4}$-transformers are found in exactly the same way). We will use the construction described in Section~\ref{sec:specialc4trans}. Let $\phi$ be an edge-bijective graph homomorphism from $C^i$ to $H_i$. For each edge $xy\in E(C^i)$, use the claim (with $W$ set to be all vertices which have been used at any point previously in the construction) to find an edge which transforms $xy\in E(C^i)$ to $\phi(xy)$ and thus obtain suitable embeddings for the vertices in $\bigcup_{e\in E(C^i)} Z^e$. It is then an easy task to greedily embed remaining vertices of the transformer (the vertices of the form $w^x$ for some $x\in V(C^i)$), since each vertex of this type has at most two neighbours previously embedded. Continuing in this way, we find edge-disjoint absorbers $A_i$ for each $H_i$ such that $|A_i|\leq m_3^3$. Let $A^*:=\bigcup_{i=1}^p A_i$ and note that $|A^*|\leq pm_3^3\leq 2^{m_3^2}$.
\end{proofof}

\subsection{Type 1 extremal}\label{sec:type1}
In this section, we will prove Lemma~\ref{lem:t1ort2} for graphs which are type~1 extremal. The next result takes any graph $G$ which is type~1 extremal and partitions its vertices into sets $A$, $B$ and $C$ so that each vertex has many neighbours in two of the parts.

\begin{prop}\label{prop:setup}
Let $n,m\in \N$ such that
$1/n\ll 1/m\ll 1$.
Let $G$ be a graph on $n$ vertices with $\delta(G)\geq 2n/3-m$. Suppose $G$ is $m$-extremal of type 1. Then there exists a partition $A, B, C$ of $V(G)$ satisfying:
\begin{enumerate}[\rm(P1)]
\item for all $v\in A$, $d_G(v, A), d_G(v, B)\geq 5n/18$;\label{prop:setup:1}%
\COMMENT{$5n/18$ bigger than $2/3 \times n/3$.}
\item for all $v\in C$, $d_G(v, B), d_G(v, C)\geq 5n/18$;\label{prop:setup:2}
\item for all but at most $3m$ vertices $v\in A$, $d_G(v, A), d_G(v, B)\geq n/3-6m$;\label{prop:setup:3}
\item for all but at most $3m$ vertices $v\in C$, $d_G(v, B), d_G(v, C)\geq n/3-6m$;\label{prop:setup:4}
\item for all $v\in B$, $d_G(v, A), d_G(v, C)\geq n/50$;\label{prop:setup:5}
\item for all but at most $50m$ vertices $v\in B$, $d_G(v, A), d_G(v, C)\geq 5n/18$;\label{prop:setup:5*}
\item $n/3-5m\leq |A|,|B|, |C|\leq n/3+3m$.\label{prop:setup:6}
\end{enumerate}
\end{prop}

\begin{proof}
Since $G$ is $m$-extremal of type 1, there exist disjoint sets $A_1, C_1\subseteq V(G)$ such that $|A_1|, |C_1|=\lceil n/3\rceil-m$ and  $e_{G}(A_1,C_1)=0$. Let $B_1:=V(G)\setminus (A_1\cup C_1)$. Since $\delta(G)\geq 2n/3-m$, for all $v\in A_1$, $d_G(v,A_1)\geq n/3-3m$ and $d_G(v,B_1)\geq n/3$. Likewise, for all $v\in C_1$, $d_G(v,C_1)\geq n/3-3m$ and $d_G(v,B_1)\geq n/3$.

Let $B_{C}$ consist of all vertices $v$ in $B_1$ such that $d_G(v,A_1)<n/50$. By considering $e_{G}(A_1, B_1)$, we obtain the following bound.
$$|A_1|n/3\leq |B_{C}|n/50+(n/3+2m-|B_{C}|)|A_1|$$
which gives
\begin{equation*}
|B_{C}|\leq \frac{2m|A_1|}{|A_1|-n/50}\leq \frac{2m|A_1|}{9|A_1|/10}\leq 3m.
\end{equation*}%
\COMMENT{$n/50\leq |A_1|/10$}
Similarly, defining $B_{A}$ to consist of all vertices $v$ in $B_1$ such that $d_G(v,C_1)<n/50$, we get $|B_{A}|\leq 3m$.
Note that $B_{A}\cap B_{C}=\emptyset$. In exactly the same way,  we can show that for all but at most $$2\cdot\frac{2m|A_1|}{|A_1|-5n/18}\leq 50m$$%
\COMMENT{$|A_1|-5n/18\geq |A_1|/12$}
vertices $v\in B$, $d_G(v, A), d_G(v, C)\geq 5n/18$.
Set $A:=A_1\cup B_{A}$, $C:=C_1\cup B_{C}$ and $B:=B_1\setminus (B_{A}\cup B_{C})$. Properties (P\ref{prop:setup:1})--(P\ref{prop:setup:6}) are satisfied.
\end{proof}

The next result refines this partition and covers all atypical edges by copies of $C_4$ to leave a dense graph with a well-defined structure.

\begin{prop}\label{prop:setup2}
Let $n,m\in \N$ such that
$1/n\ll 1/m\ll 1$.
Let $G$ be a $C_4$-divisible graph on $n$ vertices with $\delta(G)\geq 2n/3-1$. Suppose that there exists a spanning subgraph $G'$ of $G$ such that $\delta(G')\geq 2n/3-m$ and $G'$ is $m$-extremal of type 1. Then there exists $G''\subseteq G$ and a partition $A, B, C$ of $V(G'')$ satisfying:
\begin{enumerate}[\rm(Q1)]
\item $e_{G''}(A)$ and $e_{G''}(C)$ are even; \label{prop:setup2:1}
\item $G''\subseteq G[A]\cup G[C]\cup G[B,A\cup C]$ and $G-G''$ has a $C_4$-decomposition;\label{prop:setup2:2}
\item for all $v\in A$, $d_{G''}(v, A), d_{G''}(v, B)\geq n/4$;\label{prop:setup2:3}%
\COMMENT{$2/3\times n/3\leq n/4\leq 5n/18$}
\item for all $v\in B$, $d_{G''}(v, A), d_{G''}(v, C)\geq n/4$;\label{prop:setup2:5}
\item for all $v\in C$, $d_{G''}(v, B), d_{G''}(v, C)\geq n/4$;\label{prop:setup2:4}
\item $n/3-55m\leq |A|, |B|, |C|\leq n/3+3m$.\label{prop:setup2:6}
\end{enumerate}
\end{prop}

\noindent Note that we do not require $G''$ to be spanning.

\begin{proof}
First apply Proposition~\ref{prop:setup} to $G'$ to find a partition $A, B, C$ of $V(G)$ satisfying (P\ref{prop:setup:1})--(P\ref{prop:setup:6}). Suppose that $e_G(A, C)+e_G(B)=0$. It is clear that taking $G''$ as $G$ with the partition $A,B,C$ will satisfy (Q\ref{prop:setup2:2})--(Q\ref{prop:setup2:6}). We must check (Q\ref{prop:setup2:1}). Since $N_G(x)\subseteq A\cup B$ for all $x\in A$ and so on, 
\begin{equation}\label{eq:veryexcept}
|A|+|B|-1, |A|+|C|, |B|+|C|-1\geq \delta(G)
\end{equation}
which implies that
$2n=2(|A|+|B|+|C|)\geq 3\delta(G)+2$
and $\delta(G)\leq (2n-2)/3$. Note that $n\not\equiv 0\mod 3$, otherwise $\delta(G)\geq 2n/3$ since $2n/3-1$ is odd and $G$ is $2$-divisible.
We can show that $n\not\equiv 2\mod 3$ either, else $\delta(G)\geq \lceil 2n/3\rceil -1=(2n-1)/3$.
Thus $n=3N+1$ for some $N\in \N$ and $\delta(G)=2N$. The inequalities in \eqref{eq:veryexcept} must be satisfied with equality, else $|A|+|B|+|C|>n$.
Hence $|A|=|C|=N$ and $|B|=N+1$; the graphs $G[A]$, $G[B,A\cup C]$ and $G[C]$ are complete and $G$ is $2N$-regular. If $e_G(A)=e_G(C)=\binom{N}{2}$ is odd, it is easy to check that $N\equiv 2,3\mod 4$. But then $e(G)=N(3N+1)$ is not divisible by four which contradicts $G$ being $C_4$-divisible. Hence (Q\ref{prop:setup2:1}) is also satisfied.

Let us assume then that $e_G(A, C)+e_G(B)>0$.
Our first step will be to cover all edges inside $B$ and between $A$ and $C$ using copies of $C_4$. We begin by reducing the maximum degree in $G[A, C]\cup G[B]$. Choose any edge $xy \in E_G(A, C)\cup E_G(B)$, we will protect this edge for the time being since we might need it later on. Let $G_0:=(G[A, C]\cup G[B])-\{xy\}$.
Let $\eta$ be chosen such that $1/m\ll\eta\ll 1$. The Erd\H{o}s-Stone theorem allows us to greedily remove copies of $C_4$ from $G_0$ until at most $\eta n^2$ edges remain. Let $\cF_0$ denote this collection of edge-disjoint copies of $C_4$ and let $G_1:=G_0-\bigcup \cF_0$ with $e(G_1)\leq \eta n^2$.

We say that a vertex $v$ is \emph{bad} if $d_{G_1}(v)\geq \eta^{1/2}n$. Note that $G$ contains at most $2\eta n^2/(\eta^{1/2}n)=2\eta^{1/2}n$ bad vertices.
Let $B'\subseteq B$ consist of all the vertices $v\in B$ such that $d_{G}(v, A)< 5n/18$ or $d_{G}(v,C)<5n/18$. Then $|B'|\leq 50m$ by (P\ref{prop:setup:5*}).
For each bad vertex $v$, let $S_v\subseteq N_{G_1}(v)$ be a set of vertices of maximal size such that $|S_v|$ is even, no vertex in $S_v$ is bad and $S_v\cap B'=\emptyset$. 
Note that each vertex appears in at most $2\eta^{1/2}n$ sets $S_v$. Pair up the vertices in each $S_v$ arbitrarily. Our aim is to find a path of length two between each pair in $G_2:=G-(G[A, C]\cup G[B]\cup \{xy\})$. In total we have to find at most $\eta n^2/2$ paths. Note that each pair in $S_v$ has at least $n/9$ common neighbours in $G_2$ (for $S_v$ where $v\in B$, it is important that $S_v\subseteq B\setminus B'$). This allows us to greedily embed the paths so that each vertex is used at most $\eta^{1/3} n/3$ times. Write $\cF_1$ for the edge-disjoint collection of copies of $C_4$ formed by taking $\bigcup G[v\cup S_v]$ together with these paths.
Let $G_3:=G-\bigcup (\cF_0\cup \cF_1)$. We have:
\begin{enumerate}[(a)]
\item for all $v\in V(G_3)$, $d_{G_3}(v)\geq d_{G_2}(v)-\eta^{1/3}n$;\label{item:g3deg}%
\COMMENT{CAREFUL!!!!! vertices with high degree in a strange direction could lose a lot of neighbours so not true with $G$ replacing $G_2$}
\item $\Delta (G_3[B]), \Delta(G_3[A, C])\leq \eta^{1/3} n$;\label{item:g3delta}
\item $1=|\{xy\}|\leq e_{G_3}(A, C)+e_{G_3}(B)\leq \eta n^2+1$.\label{item:someedges}
\end{enumerate}

We make the following observation
\begin{equation}\label{eq:congruence}
e_{G_3}(A, C)+e_{G_3}(B)\equiv e_{G_3}(A)+e_{G_3}(C)\mod 2.
\end{equation}
To see \eqref{eq:congruence}, note that $G_3$ is $C_4$-divisible since it was obtained by removing edge-disjoint copies of $C_4$ from $G$. In particular, this means that $G_3$ is $2$-divisible and so $e_{G_3}(A\cup C, B)$ is even. Since $e(G_3)$ is also even, the result follows.

We use \eqref{eq:congruence} to cover all remaining edges in $E_{G_3}(A, C)\cup E_{G_3}(B)$, at the same time ensuring we leave an even number of edges behind in each of $A$ and $C$.
If $e_{G_3}(C)$ is odd, then assign one edge from $E_{G_3}(A, C)\cup E_{G_3}(B)$ to $C$ (we use \eqref{item:someedges} to ensure that this edge exists) and the remainder to $A$. Otherwise, assign all edges from $E_{G_3}(A, C)\cup E_{G_3}(B)$ to $A$. Find a copy of $C_4$ covering each $e\in E_{G_3}(A, C)\cup E_{G_3}(B)$ of the following form (here we say that a cycle has the form $X_1X_2X_3X_4$ to indicate that the cycle visits vertices in $X_1$, $X_2$, $X_3$ and $X_4$ in this order):
\begin{itemize}
	\item $BBXX$, if $e\in E_G(B)$ and $e$ is assigned to $X\in \{A,C\}$;
	\item $CAAB$, if $e\in E_G(A,C)$ and $e$ is assigned to $A$;
	\item $ACCB$, if $e\in E_G(A,C)$ and $e$ is assigned to $C$.
\end{itemize}
We first check that it is possible to find cycles of these forms without using any vertex too often. The ordering of each cycle above is suggestive of the order in which its vertices should be embedded (for cycles of the form $BBXX$, choose the first vertex in $X$ to satisfy (P\ref{prop:setup:3}) or (P\ref{prop:setup:4}) in $G$, i.e., not one of the exceptional $3m$ vertices). Properties (P\ref{prop:setup:1})--(P\ref{prop:setup:6}) together with \eqref{item:g3deg} ensure that there are at least $n/100$ suitable candidates in $G_3$ for each vertex which is not an endpoint of the fixed edge $e$. In total we must find at most $\eta n^2+1$ cycles and each vertex appears in the fixed edge $e$ for at most $\eta^{1/3}n$ of these cycles, by \eqref{item:g3delta} and \eqref{item:someedges}. So it is possible to embed cycles of the required forms so that each vertex is used at most $2\eta^{1/3}n$ times. Let $\cF_2$ denote the collection of cycles thus obtained and let $G_4:=G_3-\bigcup \cF_2$. For each $v\in V(G_4)$, we have
\begin{equation}\label{eq:g4deg}
d_{G_4}(v)\geq d_{G_2}(v)-5\eta^{1/3}n.
\end{equation}

We now check that removing these cycles has the desired effect.  Observe that any edge which is assigned to $A$ forms a $C_4$ which uses one edge from $E_{G_3}(A)$ and no edges from $E_{G_3}(C)$. The same statement holds with $A$ and $C$ swapped. If $e_{G_3}(C)$ is odd, deleting the cycles in $\cF_3$ will remove one edge from $E_{G_3}(C)$ leaving $e_{G_4}(C)$ even. If $e_{G_3}(C)$ is even, no edges were assigned to $C$ so $e_{G_4}(C)$ remains even. To see that $e_{G_4}(A)$ will also be even, we note that \eqref{eq:congruence} implies that the number of edges assigned to $A$ was congruent to $e_{G_3}(A) \mod 2$.

Lastly, we cover all edges incident to vertices in $B'$ (so that we can ignore $B'$). Take each vertex $v\in B'$ and pair its neighbours up arbitrarily. Find a path of length two between each pair in $G_4[A\cup C,B]$ (each such path will form a copy of $C_4$ which covers two edges incident at $v$). By (P\ref{prop:setup:1}), (P\ref{prop:setup:2}) and \eqref{eq:g4deg}, any pair of vertices in $A\cup C$ has at least $n/10$ common neighbours in $B$ and, in total, we are required to find at most $|B'|n/2\leq 25mn$ paths. So we can find a collection $\cF_3$ of edge-disjoint copies of $C_4$ which covers all edges incident at $B'$ and uses each vertex in $V(G)\setminus B'$ at most $\eta n$ times. Let $B'':=B\setminus B'$ and let $G'':=(G_4-\bigcup \cF_3)\setminus B'$. It is easy to check that $G''$ with the partition $A, B'', C$ satisfies (Q\ref{prop:setup2:1})--(Q\ref{prop:setup2:6}). 
\end{proof}

Proposition~\ref{prop:setup2} takes us most of the way towards proving Lemma~\ref{lem:t1ort2} for graphs of type~1. All that remains is to show that the graphs $G''[A]$, $G''[C]$, $G''[A,B]$ and $G''[B,C]$ can be made to be $C_4$-divisible and then to decompose these using Theorems~\ref{thm:witheps}~and~\ref{thm:bipartiteversion}.

\begin{proofof}{Lemma~\ref{lem:t1ort2}\eqref{type1}}
Apply Proposition~\ref{prop:setup2} to $G$ to find $G''\subseteq G$ and a partition $A, B, C$ of $V(G'')$ satisfying properties (Q\ref{prop:setup2:1})--(Q\ref{prop:setup2:6}). We begin by making the graphs $G''[A]$ and $G''[C]$ $C_4$-divisible. Towards this aim, let $A'\subseteq A$ consist of all vertices $v\in A$ such that $d_{G''}(v,A)$ is odd. Clearly, $|A'|$ is even. Pair up the vertices in $A'$ arbitrarily. For each pair $a_1, a_2$, find a copy of $C_4$ of the form $a_1 A a_2 B$ in $G''$. Note that on removing a copy of $C_4$ of this form, $a_1$ and $a_2$ will both have even degree in $A$ and the degree of the third vertex in $A$ is reduced by two so its parity will not be changed. Do the same for the vertices in $C$ (finding cycles of the form $c_1Cc_2B$). Note that in total we must find at most $n/2$ copies of $C_4$. Properties (Q\ref{prop:setup2:3}), (Q\ref{prop:setup2:4}) and (Q\ref{prop:setup2:6}) imply that each pair has at least $n/10$ common neighbours in the required vertex classes, so we can avoid using any vertex more than $20$ times. Write $\cF_1$ for this collection of copies of $C_4$ and let $G_1:=G''-\bigcup \cF_1$. Now every vertex in $G_1[A]$ and $G_1[C]$ has even degree.

We also require the number of edges in $G_1[A]$ and in $G_1[C]$ to be divisible by four. We know already that the number of edges will be even (from (Q\ref{prop:setup2:1}) and the fact that $\cF_1$ uses an even number of edges from both $G''[A]$ and $G''[C]$). Say that $e_{G_1}(A)\equiv 2 \mod 4$. We can fix this by removing a graph $F$ consisting of three edge-disjoint copies of $C_4$ which take the following form:
$a_1Aa_2B, a_2Aa_3B, a_1Aa_3B$ where $a_1, a_2, a_3\in A$.
Note that $F[A]$ is a copy of $C_6$, so removing $F$ does not cause the degree of any vertex in $G_1[A]$ to become odd. We can remove a similar graph if $e_{G_1}(C)$ not divisible by four. We obtain a graph $G_2$ such that $G_2[A]$ and $G_2[C]$ are $C_4$-divisible. It follows from (Q\ref{prop:setup2:3}), (Q\ref{prop:setup2:4}) and (Q\ref{prop:setup2:6}) that $$\delta(G_2[A]), \delta(G_2[C])\geq n/4-50\geq (2/3+1/100)|A|, (2/3+1/100)|C|.$$ So we can apply Theorem~\ref{thm:witheps} to find $C_4$-decompositions $\cF_A$ and $\cF_C$ of $G_2[A]$ and $G_2[C]$, respectively. Let $G_3:=G_2- \bigcup (\cF_A\cup \cF_C)$.

We will now make the bipartite graphs $G_3[A,B]$ and $G_3[B,C]$ $C_4$-divisible. Note that for any $v\in A\cup C$, $d_{G_3}(v,B)$ is necessarily even. Let $B'\subseteq B$ consist of all vertices $v\in B$ such that $d_{G_3}(v, A)$ (and hence $d_{G_3}(v,C)$) is odd. Since $e_{G_3}(A,B)$ is even, $|B'|$ must also be even. Pair up the vertices in $B'$ arbitrarily. For each pair $b_1, b_2$, find a copy of $C_4$ of the form $Ab_1Cb_2$.
On removing these copies from $G_3$, we see that $b_1$ and $b_2$ now have even degree in $A$ and in $C$.
Properties (Q\ref{prop:setup2:5}) and (Q\ref{prop:setup2:6}) ensure that there are at least $n/10$ suitable candidates at each step of the embedding. Since there are fewer than $n/2$ pairs, we can choose these copies of $C_4$ so that no vertex is used more than $10$ times.
If, after removing these copies, the number of edges between $A$ and $B$ is not divisible by four then it must be congruent to $2\mod{4}$. We can correct this by removing three further edge-disjoint copies of $C_4$ of the form: $b_1Ab_2C$, $b_2Ab_3C$, $b_1Ab_3C$ where $b_1, b_2, b_3$ are distinct vertices in $B$.%
\COMMENT{choose $b_1$, $b_2$, $b_3$ first, then easy to find rest.}
Note that removing these copies of $C_4$ removes $6\equiv 2\mod 4$ edges between $A$ and $B$ but will not change the parity of $d(b_i,A)$ for any $i\in\{1,2,3\}$.
Write $\cF_2$ for the copies of $C_4$ removed in this step and let $G_4:=G_3- \bigcup \cF_2$. 
We now have $C_4$-divisible bipartite graphs $G_4[A,B]$ and $G_4[B,C]$ and $d_{G_4}(v, B)\geq n/4-100$ for all $v\in A\cup C$. Recall (Q\ref{prop:setup2:6}), which implies 
$$\dbip(G_4[A,B]), \dbip(G_4[B,C])\geq 2/3+1/100.$$ So we can use Theorem~\ref{thm:bipartiteversion} to find a $C_4$-decomposition of $G_4$. Thus we have found a $C_4$-decomposition of $G$.
\end{proofof}

\subsection{Type 2 extremal}\label{sec:type2}

In this section, we prove Lemma~\ref{lem:t1ort2} for graphs which are type~2 extremal. We begin by showing that graphs of this type closely resemble a balanced tripartite graph with high minimum degree. 

\begin{prop}\label{prop:setup'}
Let $n,m\in \N$ such that
$1/n\ll 1/m\ll 1$.
Let $G$ be a $C_4$-divisible graph on $n$ vertices. Suppose that there exists a spanning subgraph $G'$ of $G$ such that $\delta(G')\geq 2n/3-m$ and $G'$ is $m$-extremal of type 2. Then there exists $G''\subseteq G$ and a partition $A, B, C$ of $V(G'')$ satisfying:
\begin{enumerate}[\rm(R1)]
\item $|A|$, $|B|$ and $|C|$ are even;\label{prop:setup':1}%
\COMMENT{means that we only need to make degrees even in order to get $C_4$-divisibility later on (i.e. $4|e(A,B)$ etc. will be automatic)}
\item $n/3-50m\leq |A|, |B|, |C|\leq n/3+2m$;\label{prop:setup':2}
\item $G- G''$ has a $C_4$-decomposition;\label{prop:setup':3}
\item for each $X\in \{A,B,C\}$ and each $v\in V(G'')\setminus X$, we have $d_{G''}(v, X)\geq n/4$.\label{prop:setup':4}
\end{enumerate}
\end{prop}

\noindent Again, $G''$ is not necessarily spanning.

\begin{proof}
Since $G'$ is $m$-extremal of type 2, there exist disjoint sets $A_1, B_1\subseteq V(G)$ such that $|A_1|, |B_1|= \lceil n/3\rceil-m$ and  $e_{G'}(A_1)=e_{G'}(B_1)=0$. Let $C_1:=V(G)\setminus (A_1\cup B_1)$. For all $v\in A_1$, $d_G(v,B_1)\geq n/3-3m$ and $d_G(v,C_1)\geq n/3-1$ since $\delta(G')\geq 2n/3-m$. Likewise, for all $v\in B_1$, $d_G(v,A_1)\geq n/3-3m$ and $d_G(v,C_1)\geq n/3-1$.

Let $C_{1,A}$ consist of all vertices $v\in C_1$ such that $d_G(v,A_1)<5n/18$. By considering $e_{G'}(A_1, C_1)$, we obtain the following bound.
$$|A_1|(n/3-1)\leq |C_{1,A}|5n/18+(n/3+2m-|C_{1,A}|)|A_1|$$
which gives
\begin{equation*}
|C_{1,A}|\leq \frac{(2m+1)|A_1|}{|A_1|-5n/18}\leq \frac{(2m+1)|A_1|}{|A_1|/12}\leq 25m.
\end{equation*}
Similarly, defining $C_{1,B}$ to consist of all vertices $v$ in $C_1$ such that $d_G(v,B_1)<5n/18$, we get $|C_{1,B}|\leq 25m$. Choose at most one further vertex from each of $A_1$, $B_1$ and $C_1\setminus (C_{1,A}\cup C_{1,B})$ so that $|A_1|$, $|B_1|$ and $|C_1\setminus (C_{1,A}\cup C_{1,B})|$ are made even by their removal. Let $U$ be the set which is formed by adding these vertices to $C_{1,A}\cup C_{1,B}$. Then $|U|\leq 50m+3$.

Since any pair of vertices in $G$ has at least $n/4$ common neighbours, we can easily find a collection of edge-disjoint copies of $C_4$ which covers all edges incident at $U$ and uses each vertex in $V(G)\setminus U$ at most $m^2$ times.%
\COMMENT{Enumerate the vertices in $U$ and consider each vertex in $u_i$ turn. Cover all edges incident at $u_i$ by pairing up its neighbours arbitrarily and finding edge-disjoint paths of length two between the pairs. Delete the copies of $C_4$ thus obtained and repeat the process for $u_{i+1}$. We can do this in such a way that the each vertex in $V(G)\setminus U$ is used in at most $m^2$ copies of $C_4$.
Need to find $\leq 51mn/2$ paths. Number of vertices chosen as interior vertex on a path at least $m^2/2$ times $\leq 51mn/m^2n=51/m\ll n/3$. So do not need to use any vertex more than $m^2$ times.}
Write $\cF$ for this collection of copies of $C_4$. Let $G'':=(G-\bigcup \cF)\setminus U$. Together with the partition $A:=A_1\setminus U$, $B:=B_1\setminus U$ and $C:=C_1\setminus U$, this graph satisfies (R\ref{prop:setup':1})--(R\ref{prop:setup':4}).%
\COMMENT{$|C|\geq (n/3+2m-2)-(50m+1)\geq n/3-50m$.}
\end{proof}

We now complete the proof of Lemma~\ref{lem:t1ort2}. The idea is to cover all atypical edges to leave behind a tripartite graph with vertex classes $A,B,C$ and high minimum degree. A little more work produces a graph such that each pair of vertex classes induces a $C_4$-divisible bipartite graph which we can decompose using Theorem~\ref{thm:bipartiteversion}.

\begin{proofof}{Lemma~\ref{lem:t1ort2}\eqref{type2}}
Apply Proposition~\ref{prop:setup'} to find $G_1\subseteq G$ and a partition $A, B, C$ of $V(G_1)$ satisfying (R\ref{prop:setup':1})--(R\ref{prop:setup':4}). The next step is to cover the edges in $G_1':=G_1[A]\cup G_1[B] \cup G_1[C]$ using copies of $C_4$. Let $\eps$ be such that $1/m\ll \eps \ll 1$. Using the Erd\H{o}s-Stone theorem, we may assume that $e(G_1')\leq \eps n^2$ (by greedily removing copies of $C_4$ if necessary).
Let $U\subseteq V(G_1')$ consist of all vertices $v$ such that $d_{G_1'}(v)\geq \eps^{1/2}n$. It is clear that $|U|\leq 2\eps^{1/2}n$.
For each $v\in U$, let $S_v\subseteq N_{G'}(v)\setminus U$ be as large as possible such that $|S_v|$ is even. For each $v\in U$, arbitrarily pair up the vertices in $S_v$ and find edge-disjoint paths of length two in $G_1- G_1'$ which join the pairs (to form copies of $C_4$ together with $v$).
Properties (R\ref{prop:setup':2}) and (R\ref{prop:setup':4}) allow us to do this in such a way that each vertex is used at most $3\eps^{1/2} n$ times.%
\COMMENT{Need to find $\leq \eps n^2$ paths and each vertex is involved in at most $|U|\leq 2\eps^{1/2}n$ of these. Together (R\ref{prop:setup':2}) and (R\ref{prop:setup':4}) imply that common neighbourhood of any pair of vertices in $G_1-G_1'$ has size $>n/10$. So number of vertices used as an interior vertex more than $\eps^{1/2}n$ times $\leq \eps^{1/2}n\ll n/10$.}
Denote the set of edge-disjoint copies of $C_4$ found in this step by $\cF_1$. Let $G_2:=G_1-\bigcup \cF_1$. For each $X\in \{A,B,C\}$ and each $v\notin X$, 
\begin{align}
d_{G_2}(v, X)&\geq n/4-6\eps^{1/2}n\hspace{0.5cm} \text{and}\label{eq:setup':1}\\
\Delta (G_2[A]), \Delta (G_2[B]), \Delta (G_2[C])&\leq \max\{\eps^{1/2}n, |U|+1\} \leq 3\eps^{1/2}n.\label{eq:setup':2}
\end{align}

Now cover each remaining edge in $G_2':=G_2[A]\cup G_2[B] \cup G_2[C]$ by a copy of $C_4$ using a path of length three in $G_2- G_2'$ between its endvertices. We require at most $\eps n^2$ such paths and each vertex is an endvertex of at most $3\eps^{1/2}n$ paths, by \eqref{eq:setup':2}. There are at least $n/10$ possibilities to embed each vertex by \eqref{eq:setup':1} and (R\ref{prop:setup':2}), so we are able to find these paths so that each vertex is used at most $\eps^{1/3} n/3$ times. Remove these copies of $C_4$ and write $G_3$ for the resulting graph. Note that $A$, $B$, $C$ are independent sets in $G_3$ and, for each $X\in \{A,B,C\}$ and each $v\notin X$, 
\begin{equation}\label{eq:setup':3}
d_{G_3}(v, X)\geq n/4-\eps^{1/3} n.
\end{equation}

In this final step, we ensure that each pair of vertex classes induces a $C_4$-divisible graph. Since $G_3$ is $2$-divisible, $e_{G_3}(A, B)$ must be even.%
\COMMENT{$2$-divisible implies $e_{G_3}(C, A\cup B)$ even. We know total number of edges in $G_3$ is even.}
So there is an even number of vertices $v\in A$ such that $d_{G_3}(v, B)$ is odd (note that such $v$ will necessarily also have $d_{G_3}(v, C)$ odd since $G_3$ is $2$-divisible). Pair these odd vertices up arbitrarily and, for each pair $a_1, a_2$, remove one copy of $C_4$ of the form $a_1Ba_2C$ (this changes the parities of $d_{G_3}(a_1, B)$ and $d_{G_3}(a_2, B)$). Each pair has many common neighbours in $B$ and $C$ by \eqref{eq:setup':3}, so we can do this in such a way that each vertex is used at most ten times. Do the same for the vertices in $B$ and $C$ to obtain a graph $G_4$ such that each bipartite graph induced by a pair from $\{A, B, C\}$ is $C_4$-divisible (that the number of edges in these graphs is divisible by four follows from $2$-divisibility and (R\ref{prop:setup':1})). Each of these bipartite graphs has minimum degree at least $n/4-2\eps^{1/3} n$ and (R\ref{prop:setup':4}) implies $$\dbip(G_4[A,B]), \dbip(G_4[A,C]), \dbip(G_4[B,C])\geq 2/3+\eps.$$
So we may apply Theorem~\ref{thm:bipartiteversion} to find $C_4$-decompositions of $G_4[A, B]$, $G_4[A, C]$ and $G_4[B, C]$. This completes our $C_4$-decomposition of $G$.
\end{proofof}

\section{Even cycles of length at least eight}\label{sec:longercycle}

The aim of this section is to prove Theorem~\ref{thm:mainc2k} for even cycles of length at least eight. We will again split our argument into extremal and non-extremal cases. When $G$ is not extremal, it will satisfy an expansion property which we now describe.
Let $G$ be a graph on $n$ vertices. We define the \emph{robust neighbourhood} of a set $S\subseteq V(G)$ to be the set of vertices $R_{\nu,G}(S):=\{v\in V(G): d_G(v,S)\geq \nu n\}$.
We say that a set $S\subseteq V(G)$ is \emph{$\nu$-expanding in $G$} if $|R_{\nu,G} (S)|\geq n/2+\nu n$. We say that $G$ is a \emph{$\nu$-expander} if for every $x\in V(G)$, $N_G(x)$ is $\nu$-expanding.%
\COMMENT{This weaker definition of ``robust expander'' makes it easier to find the vortex.}
Note that every $\nu$-expander $G$ satisfies $\delta(G)\geq \nu n$.

Any graph which is not a $\nu$-expander falls into one of two classes of extremal graph.
We say that a graph $G$ on $n$ vertices is \emph{$\eps$-close to $K_{n/2}\cup K_{n/2}$} if there exists $S\subseteq V(G)$ such that $|S|=\lfloor n/2\rfloor$ and $e(S, \overline S)\leq \eps n^2$.
We say that $G$ is \emph{$\eps$-close to bipartite} if there exists $S\subseteq V(G)$ such that $|S|=\lfloor n/2\rfloor$ and $e(S)\leq \eps n^2$. The following is a weak form of Lemma~26 in \cite{lap}.

\begin{prop}[\cite{lap}]\label{prop:islike}
Let $1/n\ll \nu\ll \eps< 1$. Let $G$ be a graph on $n$ vertices with $\delta(G)\geq n/2$. Then one of the following holds:
\begin{enumerate}[\rm(i)]
\item $G$ is a $\nu$-expander;%
\COMMENT{alternatively: For every $S\subseteq V(G)$ with $|S|=\lceil n/2\rceil$, $|R_{\nu,G}(S)|\geq n/2+\nu n$;}
\item $G$ is $\eps$-close to $K_{n/2}\cup K_{n/2}$;
\item $G$ is $\eps$-close to bipartite.
\end{enumerate}
\end{prop}

The following result, which will be proved in Section~\ref{sec:expanderdecomp}, is a version of Theorem~\ref{thm:witheps} which relies on $\nu$-expansion (instead of solely the minimum degree). This result finds a $C_{2k}$-decomposition of $G$ when $G$ is a $\nu$-expander and $k\geq 4$.

\begin{theorem}\label{thm:nuexpander}
Let $k\in \N$, $k\geq 4$ and $1/n\ll \nu, 1/k$. Let $G$ be a $C_{2k}$-divisible $\nu$-expander on $n$ vertices. If $k=4$, assume further that $\delta(G)\geq n/2$. Then $G$ has a $C_{2k}$-decomposition.
\end{theorem}

Given Theorem~\ref{thm:nuexpander}, it remains to find decompositions of graphs which are close to $K_{n/2}\cup K_{n/2}$ or close to bipartite. This is achieved in the current section. Theorem~\ref{thm:mainc2k} for $k\geq 4$ will then follow directly from Proposition~\ref{prop:islike}, Theorem~\ref{thm:nuexpander}, Lemma~\ref{lem:closeclique} and Lemma~\ref{lem:closebip}.

\subsection{$G$ is close to $K_{n/2}\cup K_{n/2}$}

The next result finds a $C_{2k}$-decomposition when $G$ is close to $K_{n/2}\cup K_{n/2}$. The idea of the proof is to exploit the fact that $G$ resembles two disjoint cliques: first dealing with any unusual edges or exceptional vertices and then using Theorem~\ref{thm:witheps} to decompose the (almost) cliques.

\begin{lemma}\label{lem:closeclique}
Let $k\in \N$, $k\geq 4$ and $1/n\ll \eps\ll 1$. Suppose that $G$ is a $C_{2k}$-divisible graph on $n$-vertices and $\delta(G)\geq n/2$. Suppose further that $G$ is $\eps$-close to $K_{n/2}\cup K_{n/2}$. Then $G$ has a $C_{2k}$-decomposition.
\end{lemma}

We will prove Lemma~\ref{lem:closeclique} in stages.

\begin{prop}\label{prop:cliquepart}
Let $1/n\ll\eps\ll 1$. Suppose that $G$ is a graph on $n$ vertices with $\delta(G)\geq n/2$ which is $\eps$-close to $K_{n/2}\cup K_{n/2}$. Then there exists a partition $A,B$ of $V(G)$ such that:
\begin{enumerate}[\rm(S1)]
\item $\delta(G[A]), \delta(G[B])\geq n/5$;\label{cliquepart1}
\item for all but at most $2\sqrt\eps n$ vertices $v\in A$, $d_G(v, A)\geq n/2-2\sqrt\eps n$;\label{cliquepart2}
\item for all but at most $2\sqrt\eps n$ vertices $v\in B$, $d_G(v, B)\geq n/2-2\sqrt\eps n$;\label{cliquepart3}
\item $n/2-4\eps n\leq |A|\leq |B|\leq n/2+4\eps n$.\label{cliquepart4}
\end{enumerate}
\end{prop}

\begin{proof}
Let $S\subseteq V(G)$ such that $|S|=\lfloor n/2\rfloor$ and $e(S, \overline S)\leq \eps n^2$. Let $T:=\overline S$. For each $p\in \{11n/50, n/2-\sqrt\eps n\}$, let $S_p:=\{v\in S: d_G(v,S)\leq p\}$ and define $T_p$ similarly.
We have $|S_p|, |T_p|\leq \frac{\eps n^2}{n/2-p}$, so that
\begin{align*}
|S_{11n/50}|, |T_{11n/50}|\leq 25\eps n/7\;\text{ and }\;|S_{n/2-\sqrt\eps n}|, |T_{n/2-\sqrt\eps n}|\leq \sqrt \eps n.
\end{align*}
Let $S':=(S\setminus S_{11n/50})\cup T_{11n/50}$. Setting $A$ to be the smallest of $S'$ and $\overline{S'}$ and setting $B:=\overline A$ gives the desired partition.
\end{proof}

Before we begin decomposing $G$, we must reserve some edges between $A$ and $B$ using the following simple proposition. These edges will be used at a later stage to ensure that the graphs on $A$ and $B$ are $C_{2k}$-divisible.

\begin{prop}\label{prop:c2kcliqueedges}
Let $k\in \N$ and $1/n\ll\eps\ll 1/k$. Suppose that $G$ is a graph on $n$ vertices with $\delta(G)\geq n/2$ and $A,B$ is a partition of $V(G)$ satisfying \textup{(S\ref{cliquepart1})--(S\ref{cliquepart4})}.
Then there exist $4k$ distinct edges $e_1, \dots e_{2k}, f_1, \dots f_{2k} \in E_G(A,B)$ such that, for each $1\leq i\leq 2k$, $e_i$ and $f_i$ are vertex-disjoint and $d_G(a_i,A)\geq n/2-2\sqrt\eps n$ where $a_i:=V(e_i)\cap A$. 
\end{prop}

\begin{proof}
If $|A|<n/2$, each vertex in $A$ has at least two neighbours in $B$ so the result is clear.
So we assume that $|A|=|B|=n/2$, in which case $\delta(G[A,B])\geq 1$. Suppose that the proposition is false and let $\ell<2k$ be maximal such that $G$ contains edges $e_1, \dots e_{\ell}, f_1, \dots f_{\ell} \in E_G(A,B)$ such that, for each $1\leq i\leq \ell$, $e_i$ and $f_i$ are vertex-disjoint and $d_G(a_i,A)\geq n/2-2\sqrt\eps n$ where $a_i:=V(e_i)\cap A$.

Let $U:=\bigcup_{i=1}^{\ell}V(e_i\cup f_i)$, $A':=A\setminus U$ and $B':=B\setminus U$. Choose any vertex $a\in A'$ such that $d_G(a,A)\geq n/2-2\sqrt\eps n$ and let $b\in N_G(a,B)$. Let $a'\in A'\setminus \{a\}$ and $b'\in B'\setminus \{b\}$. If $a'b''\in E(G)$ for some $b''\neq b$, we can take $e_{\ell+1}:=ab$ and $f_{\ell+1}=a'b''$, contradicting the maximality of $\ell$. Since $d_G(a',B)\geq 1$, we must have $a'b\in E(G)$. Similarly, $ab'\in E(G)$. But then taking $e_{\ell+1}:=ab'$ and $f_{\ell+1}=a'b$ gives a contradiction.
\end{proof}

The next result covers the remaining edges between $A$ and $B$.

\begin{prop}\label{prop:c2kcliqueclearing}
Let $k\in \N$, $k\geq 4$ and $1/n\ll \eta\ll\eps\ll 1/k$. Suppose that $G$ is a graph on $n$ vertices and $A,B$ is a partition of $V(G)$ satisfying \textup{(S\ref{cliquepart1})--(S\ref{cliquepart4})}.
Suppose that $e_G(A,B)$ is even.
Then there exists a $C_{2k}$-decomposable graph $H\subseteq G$ such that $G[A,B]\subseteq H$ and $\Delta (H[A]), \Delta (H[B])\leq \eta n$.
\end{prop}

\begin{proof}
Let $G':=G[A,B]$. Use the Erd\H{o}s-Stone theorem to greedily find an $\eta^4$-approximate $C_{2k}$-decomposition $\cF$ of $G'$ and let $H_0:=\bigcup \cF$. Let $X_A:=\{v\in A:d_{G'-H_0}(v)\geq \eta^2 n\}$ and note that $|X_A|\leq \eta^4n^2/(\eta^2n)=\eta^2n$. For each vertex $x\in X_A$, pair up the vertices in $N_{G'-H_0}(x)$ arbitrarily, leaving at most one vertex unpaired. Find edge-disjoint paths of length $2k-2$ in $G[B]$ between each of the pairs (to obtain copies of $C_{2k}$ which cover all but at most one of the edges incident at $x$ in $G'-H_0$). Properties (S\ref{cliquepart1}), (S\ref{cliquepart3}) and (S\ref{cliquepart4}) allow us to find these paths so that each vertex appears as an interior vertex on at most $\eta^3n$ of the paths.%
\COMMENT{To find a path of length $2k-2$ between $x$ and $y$ in $B$, know that $d_G(x,B), d_G(y,B)\geq n/5$ and both of these neighbourhoods contain many vertices of degree $n/2-2\sqrt\eps n$. So greedily find paths, always have choice from at least $n/10$ vertices. $\eta^4n^2k/\eta^3n\ll n/10$. So we can disregard all vertices which have been used at least $\eta^3n$ times.}
Let $H_A$ denote the $C_{2k}$-decomposable graph thus obtained and repeat the process for the set of vertices $X_B:=\{v\in B:d_{G'-H_0-H_B}(v)\geq \eta^2 n\}$, obtaining a $C_{2k}$-decomposable graph $H_B$ which covers all but at most one edge incident at each $x\in X_B$.
Now $H':=H_0\cup H_A\cup H_B$ is $C_{2k}$-decomposable, $\Delta (G[A,B]-H')\leq \eta^2 n$ and 
$$\Delta (H'[A]), \Delta (H'[B])\leq 2\eta^3n+\eta^2n\leq 2\eta^2 n.$$

Since $e_G(A,B)$ and $e_{H'}(A,B)$ are even, so is $e_{G-H'}(A,B)$. Pair up the edges in $E_{G-H'}(A,B)$ arbitrarily and complete each to a copy of $C_{2k}$ as follows. If the two edges share an endvertex, in $A$ say, find a path of length $2k-2$ between their endpoints in $B$ as above. If the edges are disjoint, find paths of length $k-1\geq 3$ between their endpoints in $A$ and in $B$. Again, properties (S\ref{cliquepart1})--(S\ref{cliquepart4}) allow us to find these paths
so that they are edge-disjoint and each vertex appears as an interior vertex on at most $\eta^3n$ of the paths. Let $H''$ denote the $C_{2k}$-decomposable graph obtained in this way. We have ensured that 
$$\Delta(H''[A]), \Delta(H''[B])\leq \Delta (G[A,B]-H')+2\eta^3n\leq 2\eta^2 n.$$

Finally, let $H:=H'\cup H''$. This graph is $C_{2k}$-decomposable, $$\Delta(H[A]), \Delta(H[B])\leq 4\eta^2 n\leq \eta n$$ and $G[A,B]\subseteq H$.
\end{proof}

We combine the previous results to find a $C_{2k}$-decomposition when $G$ is $\eps$-close to $K_{n/2}\cup K_{n/2}$.

\begin{proofof}{Lemma~\ref{lem:closeclique}}
Choose a constant $\eta$ such that $1/n\ll \eta\ll \eps$. 
Apply Proposition~\ref{prop:cliquepart} to obtain a partition $A,B$ of $G$ satisfying (S\ref{cliquepart1})--(S\ref{cliquepart4}). Then apply Proposition~\ref{prop:c2kcliqueedges} to reserve edges $\cE:=\{e_1, \dots, e_{2k}, f_1, \dots, f_{2k}\}$. Let $G':=G-\bigcup \cE$ and note that $G'$ with the partition $A,B$ still satisfies (S\ref{cliquepart1})--(S\ref{cliquepart4}) of Proposition~\ref{prop:cliquepart}. Since $G$ is $2$-divisible, $e_G(A,B)$ is even, so $e_{G'}(A,B)=e_G(A,B)-4k$ is also even. So we can apply Proposition~\ref{prop:c2kcliqueclearing} to find $H\subseteq G'$ which has a $C_{2k}$-decomposition $\cF_1$ such that $G'[A,B]\subseteq H$ and $\Delta (H[A]), \Delta (H[B])\leq \eta n$.

We have covered all edges in $G[A,B]$ apart from those in $\cE$, which we will use to ensure that $(G-H)[A]$ and $(G-H)[B]$ are $C_{2k}$-divisible. To this end, let $0\leq r\leq 2k-1$ be chosen such that $e_{G-H}(A)\equiv r \mod{2k}$. We will find $2k$ copies of $C_{2k}$, each containing a pair $e_i,f_i$, as follows. For each $1\leq i\leq 2k-r$, find a path of length $2$ between the endpoints of $e_i$ and $f_i$ in $(G-H)[A]$ and a path of length $2k-4$ between the endpoints of $e_i$ and $f_i$ in $(G-H)[B]$. For each $2k-r< i\leq 2k$, find a path of length $3$ between the endpoints of $e_i$ and $f_i$ in $(G-H)[A]$ and a path of length $2k-5$ between the endpoints of $e_i$ and $f_i$ in $(G-H)[B]$. (The property $d_G(a_i,A)\geq n/2-2\sqrt\eps n$ where $a_i:=e_i\cap A$ is needed for finding the paths of length $2$.) We can choose these paths to be edge-disjoint. Let $\cF_2$ denote the copies of $C_{2k}$ thus obtained and let $G'':=G-H-\bigcup \cF_2$. We make the following important observation: $G''[A]$ and $G''[B]$ are $C_{2k}$-divisible. That these graphs are $2$-divisible is clear (they were obtained by removing edge-disjoint copies of $C_{2k}$ from a $2$-divisible graph $G$). To see that $e_{G''}(A)$ is divisible by $2k$, note that
$$e_{G''}(A)=e_{G-H}(A)-2(2k-r)-3r\equiv r-4k+2r-3r\equiv 0 \mod{2k}$$
(and $e_{G''}(B)$ is also divisible by $2k$).

Finally, note that
$$\Delta((G-G'')[A]), \Delta((G-G'')[B])\leq 2\eta n$$ and recall (S\ref{cliquepart1})--(S\ref{cliquepart4}). Let $X:=\{x: d_{G''}(x)<n/2-3\sqrt\eps n\}$.  Then $|X|\leq 4\sqrt\eps n$ and we can easily cover all edges incident at vertices in $X$ using a collection $\cF_3$ of edge-disjoint copies of $C_{2k}$ such that no vertex in $V(G'')\setminus X$ is used more than $\eps^{1/3}n$ times. Let $G''':=(G''-\bigcup \cF_3)\setminus X$. Now
$$\delta(G''')\geq n/2-3\eps^{1/3}n \geq (2/3+\eps)\cdot\max\{|A\setminus X|, |B\setminus X|\}.$$
We then find $C_{2k}$-decompositions $\cF_4$ and $\cF_5$ of $G'''[A]$ and $G'''[B]$ respectively, using Theorem~\ref{thm:witheps}.%
\COMMENT{Theorem~\ref{thm:nuexpander} would still need mindeg $n/2$ for $C_8$.}
Then $\bigcup_{i=1}^5 \cF_i$ gives a $C_{2k}$-decomposition of $G$.
\end{proofof}

\subsection{$G$ is close to bipartite}

We now consider the case when $G$ is close to bipartite. We will process the graph, covering any unusual edges or exceptional vertices with copies of $C_{2k}$ until we really are left with a dense bipartite graph. This we can decompose using Theorem~\ref{thm:bipartiteversion}.

\begin{lemma}\label{lem:closebip}
Let $k\in \N$, $k\geq 4$ and $1/n\ll \eps\ll 1$. Suppose that $G$ is a $C_{2k}$-divisible graph on $n$-vertices and $\delta(G)\geq n/2$. Suppose further that $G$ is $\eps$-close to bipartite. Then $G$ has a $C_{2k}$-decomposition.
\end{lemma}

The following proposition partitions the vertices of $G$ into an ``almost bipartite'' graph with high minimum degree.

\begin{prop}\label{prop:closebip1}
Let $k\in \N$, $k\geq 4$ and $1/n\ll \eps\ll 1$. Suppose that $G$ is a $C_{2k}$-divisible graph on $n$-vertices and $\delta(G)\geq n/2$. Suppose further that $G$ is $\eps$-close to bipartite. Then there exists $G'\subseteq G$ and a partition $A, B$ of $V(G')$ such that the following hold:
\begin{enumerate}[\rm(T1)]
	\item $\delta(G'[A,B])\geq n/3$;\label{closebip1:1}
	\item $G-G'$ has a $C_{2k}$-decomposition;\label{closebip1:2}
	\item $|A|, |B|= n/2\pm 6\sqrt\eps n$;\label{closebip1:3}
	\item $e(G'[A])+e(G'[B])<\eps n^2$.\label{closebip1:4}
\end{enumerate}
\end{prop}

\noindent Note that $G'$ is not necessarily spanning.

\begin{proof}
Let $S\subseteq V(G)$ be such that $|S|=\lfloor n/2 \rfloor$ and $e_G(S)\leq \eps n^2$. Let $T:=\overline S$ and consider the bipartite graph $G_0:=G[S, T]$. We want to transform $G_0$ into a bipartite graph whose minimum degree is as high as possible. We first modify the bipartition  $S,T$ to obtain a new bipartition $S', T'$. Let
$$X:=\{x:d_{G_0}(x)<n/2-\sqrt\eps n\}.$$
It is easy to see that
$|X|\leq 5\sqrt\eps n.$
\COMMENT{let $Y:=\{x\in S:d_{G_0}(x)<n/2-\sqrt\eps n\}$ and $Z:=\{x\in T:d_{G_0}(x)<n/2-\sqrt\eps n\}$. Any vertex in $Y$ must have at least $\sqrt\eps n$ neighbours in $S$. So $|Y|\leq 2\eps n^2/\sqrt\eps n\leq 2\sqrt\eps n$. To bound $|Z|$, observe that 
$$|S|\lceil n/2\rceil-2\eps n^2\leq e(G_0)< |Z|(n/2-\sqrt\eps n)+(\lceil n/2\rceil-|Z|)|S|$$
so $|Z|\leq 2\eps n^2/(|S|-(n/2-\sqrt\eps n))\leq 3\sqrt\eps n$.}
Let
$$X_S:=\{x\in X: d_G(x, S)< 5n/12\}\text{ and let }X_T:=X\setminus X_S.$$
Let $S':=(S\setminus X)\cup X_S$ and let $T':=\overline {S'}=(T \setminus X)\cup X_T$.
It is useful to note that:
\begin{enumerate}[(i)]
\item for any $x\in S'$, $d_G(x, T')\geq n/13$ and, if $x\in S'\setminus X$, $d_G(x, T')\geq n/2-6\sqrt\eps n$;\label{item:S'}
\item for any $x\in T'$, $d_G(x, S')\geq 5n/13$ and, if $x\in T'\setminus X$, $d_G(x, S')\geq n/2-6\sqrt\eps n$.\label{item:T'}
\end{enumerate}
Let $X_0:=\{x\in X: d_G(x, S)\text{ and } d_G(x, T)< 5n/12\}$. Since $X_0\subseteq X_S$, the vertices in $X_0$ have all been assigned to $S'$ but they do not naturally belong to either side of the partition so we will cover all edges incident at these vertices in the next step.

Choose any vertex $x\in X_0$. Suppose that $d_G(x,S')$ is odd. Note that $e_G(S', T')$ is even (since $G$ is $2$-divisible).%
\COMMENT{implies every component of $G$ is Eulerian.}
This means that $e_G(S')+e_G(T')$ is also even (recall that the number of edges in $G$ is divisible by $2k$). In particular, there must be an edge $uv\in E(S')\cup E(T')$ which is not incident at $x$. Let $y\in N_G(x,S')\setminus (X\cup \{u,v\})$.  We now find a copy of $C_{2k}$ which uses both $xy$ and $uv$. If $u,v\in S'$, note that $|N_G(y)\cap N_G(u)|\geq n/15$ by \eqref{item:S'}, so we can find a path of length two from $u$ to $y$. We also find a path of length $2k-4\geq 4$ between $x$ and $v$. At each stage, we can choose from at least $n/20$ vertices. This gives a copy of $C_{2k}$. We proceed in a similar way when $u,v\in T'$.%
\COMMENT{find paths of length $k-1\geq 3$ between $x$ and $u$ and between $y$ and $v$.}
We may now assume that $d_G(x,S')$ is even. Pair up the neighbours of $x$ arbitrarily and find edge-disjoint paths of length $2k-2$ between each pair in $G[S', T']$ (to obtain edge-disjoint copies of $C_{2k}$). Remove all copies of $C_{2k}$ obtained in this way from $G$. Repeat this process for the remaining vertices in $X_0$ and write $\cF_1$ for the collection of copies of $C_{2k}$ thus obtained. Let $G_1:=G-\bigcup \cF_1$. At the end of this process, we may assume that each vertex in $V(G)\setminus X_0$ has been used in at most $\eps^{1/3}n/2$ copies of $C_{2k}$.%
\COMMENT{Always able to choose from $n/20$ vertices. Need to find $|X_0|n\leq 5\sqrt\eps n^2$ paths and each vertex is at the end of at most $5\sqrt\eps n$ of these. So can avoid using any vertex more than $\eps^{1/3}n/2$ times.}

Let $A:=S'\setminus X_0$, $B:=T'$. Observe that $|A|, |B|=n/2\pm 6\sqrt\eps n$ and $$\delta(G_1[A,B])\geq 5n/13-\eps^{1/3}n\geq n/3.$$
Using the Erd\H{o}s-Stone theorem, we find an $\eps$-approximate $C_{2k}$-decomposition $\cF_2$ of $G_1[A]\cup G_1[B]$. Letting $G':=G-\bigcup (\cF_1\cup \cF_2)$ completes the proof.
\end{proof}

We use this proposition to prove Lemma~\ref{lem:closebip}.

\begin{proofof}{Lemma~\ref{lem:closebip}}
Apply Proposition~\ref{prop:closebip1} to find $G', A, B$ satisfying (T\ref{closebip1:1})--(T\ref{closebip1:4}). Let $\cF_1$ be a $C_{2k}$-decomposition of $G-G'$.
Let $A':=\{x\in A: d_{G'}(x, A)\geq \sqrt\eps n\}$ and define $B'$ similarly. Note that $|A'\cup B'|\leq 2\sqrt\eps n$ by (T\ref{closebip1:4}). Take each vertex $x\in A'$ in turn and split $N_{G'}(x, A)$ into pairs (leaving at most one vertex). Use (T\ref{closebip1:1}) and (T\ref{closebip1:3}) to find a path of length $2k-2$ between each pair in $G'[A,B]$ to obtain a copy of $C_{2k}$ together with $x$. Carry out this process for the remaining edges at each  remaining vertex in $A'$. Do the same for the vertices in $B'$. We may carry out this process so that each vertex appears in at most $\eps^{1/3}n$ of the paths. Write $\cF_2$ for the collection of copies of $C_{2k}$ obtained in this way and let $G_1:=G'-\bigcup \cF_2$. We have $\Delta(G_1[A]), \Delta(G_1[B])<\eps^{1/3}n$ and 
\begin{equation}\label{eq:g1bip}
\delta(G_1[A, B])\geq n/3-2\eps^{1/3}n.
\end{equation}

We now cover the remaining edges in $E_{G_1}(A)\cup E_{G_1}(B)$. There are an even number of these so we can pair them up arbitrarily. We use \eqref{eq:g1bip} to find paths of even length at least two between any two vertices in the same class and paths of odd length at least three between any two vertices in different classes. At each step we have a choice of at least $n/10$ vertices so we are able to find edge-disjoint copies of $C_{2k}$ (by finding paths of suitable length between the endpoints of each pair of edges) so that each pair of edges is covered and no vertex appears in more than $2\eps^{1/3}n$ of the cycles. Write $\cF_3$ for the collection of copies of $C_{2k}$ obtained in this step. The graph $G_2:=G_1-\bigcup \cF_3$ is $C_{2k}$-divisible and bipartite with vertex classes $A$ and $B$ of size $n/2\pm 6\sqrt\eps n$. Furthermore, $\delta(G_2)\geq n/3-6\eps^{1/3}n$ so $\dbip(G_2)\geq 1/2+\eps$. Thus $G_2$ has a $C_{2k}$-decomposition $\cF_4$ by Theorem~\ref{thm:bipartiteversion}. Together, $\bigcup_{i=1}^4\cF_i$ gives a $C_{2k}$-decomposition of $G$.
\end{proofof}

\section{Decompositions of bipartite graphs}\label{sec:bipartite}

In this section we prove Theorem~\ref{thm:bipartiteversion}, the bipartite version of  Theorem~\ref{thm:witheps}. Theorem~\ref{thm:bipartiteversion} finds a $C_{2k}$-decomposition of $G$ when $G$ is bipartite and has high minimum degree. We used this result to prove Theorem~\ref{thm:mainc2k} earlier on. The proof closely follows the iterative absorption argument of \cite{stef}, thus we omit some of the details.

We require the following definition, a bipartite version of the vortices considered in Section~\ref{sec:c4}. 
Let $G=(A,B)$ be a bipartite graph. A \emph{$(\delta,\mu,m)$-vortex respecting $(A,B)$} in $G$ is a sequence $U_0 \supseteq U_1 \supseteq \dots \supseteq U_\ell$ such that
\begin{itemize}
\item $U_0=V(G)$;
\item $|U_i\cap X|=\lfloor \mu|U_{i-1}\cap X|\rfloor$ for all $1\leq i\leq \ell$ and each $X\in \{A,B\}$, and $|U_\ell|= m$;
\item $d_G(x,U_i\cap X) \ge \delta |U_i\cap X|$, for all $1\leq i\leq \ell$, each $X\in \{A,B\}$ and all $x \in U_{i-1}\setminus X$.
\end{itemize}
The following observation guarantees a vortex in $G$. It is proved by repeatedly applying the Chernoff bound given by Lemma~\ref{lem:chernoff} (for more details, see%
\COMMENT{Appendix~\ref{append:bipcover}}
the proof of Lemma~4.3 in~\cite{stef}).

\begin{lemma} \label{lem:getvortex:bip}
Let $0\leq \delta\leq 1$ and $1/m'\ll \mu<1$. Suppose that $G=(A,B)$ is a bipartite graph with $m'\leq |A|\leq |B|\leq 2|A|$ and $\dbip \ge \delta$. Then $G$ has a $(\delta-\mu,\mu,m)$-vortex respecting $(A,B)$ for some $2\lfloor \mu m' \rfloor \le m \le m'$.
\end{lemma}

The idea is to use the following result to cover almost all of the edges in $G$ leaving only a small (very restricted) remainder which can be dealt with using the absorbers given by Lemma~\ref{lem:abs:bip}.

\begin{lemma} \label{lem:nearoptimal:bip}
Let $k\in \N$, $k\geq 2$ and let $1/m \ll \mu\ll 1/k$. Let $G=(A,B)$ be a bipartite $2$-divisible graph with $n\leq |A|,|B|\leq 2n$ and $\dbip \ge 1/2 + 3\mu$. Let $U_0 \supseteq U_1 \supseteq \dots \supseteq U_\ell$ be a $(1/2+4\mu,\mu,m)$-vortex respecting $(A,B)$ in $G$. Then there exists $H_\ell\subseteq G[U_\ell]$ such that $G-H_\ell$ is $C_{2k}$-decomposable.
\end{lemma}

We will prove Lemma~\ref{lem:nearoptimal:bip} in Section~\ref{sec:bip:most}. Theorem~\ref{thm:bipartiteversion} then follows directly from these results.

\begin{proofof}{Theorem~\ref{thm:bipartiteversion}}(Assuming Lemma~\ref{lem:nearoptimal:bip}.)
Let $m,n_0\in \N$ and $\mu$ be such that 
$$1/n_0\ll 1/m\ll \mu\ll \eps, 1/k.$$
Apply Lemma~\ref{lem:getvortex:bip} to find $U_0\supseteq U_1\supseteq \dots \supseteq U_\ell$, a $(\delta_k+\eps/2, \mu, m)$-vortex respecting $(A,B)$ in $G$.

Let $G_1:=G-G[U_1]$. We have $\dbip(G_1)\geq \delta_k+\eps/2$. Apply Lemma~\ref{lem:abs:bip} to $G_1$ with $U_\ell$ playing the role of $U$ to find an absorber $A^*\subseteq G_1$ as in the lemma. We have $\Delta (A^*)\leq |A^*|\leq 2^{m^2}$, so $U_0\supseteq U_1\supseteq \dots \supseteq U_\ell$ is a $(\delta_k+4\mu, \mu, m)$-vortex respecting $(A,B)$ in $G^*:=G-A^*$ and $\dbip(G^*)\geq 1/2+3\mu$. Then apply Lemma~\ref{lem:nearoptimal:bip} to $G^*$ to find $H_\ell\subseteq G^*[U_\ell]$ such that $G^*-H_\ell$ has a $C_{2k}$-decomposition. Observing that $A^*\cup H_\ell$ has a $C_{2k}$-decomposition (by Lemma~\ref{lem:abs:bip}) completes the proof.
\end{proofof}

\subsection{Proving Lemma~\ref{lem:nearoptimal:bip}}\label{sec:bip:most}

First, we state some useful results. We require the following simple proposition on decompositions of cliques. It is a special case of Wilson's Theorem and is proved very easily (see \cite{stef}, for example).

\begin{prop} \label{prop:clique dec}
Let $p$ be prime. Then for every $k\in \N$, $K_{p^k}$ has a $K_p$-decomposition. 
\end{prop}

We use the next result to find an approximate $C_{2k}$-decomposition of $G$ and maintain some control over the number of edges incident at any vertex in a given set $X$.

\begin{lemma} \label{lem:careofbad:bip}
Let $k\in \N$, $k\geq 2$ and $1/n \ll \eta \ll \eps,1/k$. Suppose that $G=(A,B)$ is a bipartite graph with $n\leq |A|,|B|\leq 4n$ and $\dbip(G)\geq 1/2+\eps$. Let $X\subseteq V(G)$ of size at most $\eta^{1/2} n$. Then there exists $H\subseteq G$ such that $G-H$ is $C_{2k}$-decomposable, $Y:=\{x\in V(G):d_H(x)>\eta n\}$ has size at most $\eta n$ and $X\cap Y=\emptyset$.
\end{lemma}

\begin{proof}
The first step is to cover the edges in $G[X]$ by edge-disjoint copies of $C_{2k}$. That is, for each edge $xy\in E_G(X)$, find a path of length $2k-1$ between the $x$ and $y$ in $G-G[X]$ ($x$ and $y$ lie in different vertex classes so any path between them is necessarily odd). In total we must find at most $\eta n^2$ paths. Since $\dbip(G-G[X])\geq 1/2+3\eps/4$, we may choose these paths to be edge-disjoint and use each vertex at most $\eta^{1/3}n$ times. These paths, together with $E_G(X)$ give an edge-disjoint collection $\cF_X$ of copies of $C_{2k}$ with $\Delta(\bigcup \cF_X)\leq 2\eta^{1/3}n$.

Consider the graph $G':=G\setminus \bigcup \cF_X$. Our next step is to cover all but at most one of the remaining edges incident at each vertex in $X$. For each $x\in X$, pair up the vertices in $N_{G'}(x)$, leaving at most one vertex. Note that both vertices in any pair lie in the same vertex class. Since $\dbip(G'\setminus X)\geq 1/2+\eps/2$, we can find edge-disjoint paths of (even) length $2k-2$ between each pair in $G'\setminus X$.%
\COMMENT{We are required to find at most $4\eta^{1/2}n^2$ paths and each vertex contained in at most $\eta^{1/2}n$ pairs.}
Each path combines with two edges incident at $X$ to form a copy of $C_{2k}$.
Thus we obtain a collection $\cF_X'$ of edge-disjoint copies of $C_{2k}$ which, together with $\cF_X$, cover all but at most one edge incident at each $x\in X$.

Let $H':=G-\bigcup (\cF_X\cup \cF_X')$ and note that $d_{H'}(x) \leq 1$ for all $x\in X$. Use the Erd\H{o}s-Stone theorem to greedily find an $\eta^3$-approximate $C_{2k}$-decomposition of $H'$ which we will denote by $\cF$. Let $H:=H'-\bigcup \cF$ and note that $G-H$ has a $C_{2k}$-decomposition given by $\cF_X\cup \cF_X'\cup \cF$. If $Y:=\{x\in V(G):d_H(x)>\eta n\}$, then $|Y|\leq 2e(H)/(\eta n)\leq \eta n$ and $X\cap Y=\emptyset$.
\end{proof}

We use Lemma~\ref{lem:careofbad:bip} to prove the following result which finds a $C_{2k}$-decomposition of $G$ so that every vertex has low degree in the remainder.

\begin{lemma} \label{lem:boundmaxdegree:bip}
Let $k\in \N$, $k\geq 2$ and $1/n \ll \eps,1/k$.
Let $G=(A,B)$ be a bipartite graph with $n\leq |A|,|B|\leq 3n$ and $\dbip(G)\geq 1/2+\eps$. Then $G$ has an approximate $C_{2k}$-decomposition $\cF$ such that $\Delta(G-\bigcup \cF)\leq \eps n$.
\end{lemma}

\begin{proof}
Choose $s, t\in \N$ and $\eta>0$ such that 
$$1/n\ll \eta \ll 1/s  \ll 1/t \ll \eps,1/k$$
and $K_s$ has a $K_t$-decomposition ($s$ and $t$ exist by Proposition~\ref{prop:clique dec}). 
Let $\cP=\{V_1,\dots,V_s\}$ be a partition of $V(G)$ satisfying the following for all $1\leq i\leq s$ and each $X\in \{A,B\}$:
\begin{enumerate}[\rm(i)]
\item $|V_i\cap X|= \lfloor |X|/s\rfloor$ or $\lceil |X|/s\rceil$;\label{rand1*:bip}
\item $d_G(x,V_i\cap X) \ge (1/2+2\eps/3)|V_i\cap X|$ for all $x\in V(G)\setminus X$.\label{rand2*:bip}
\end{enumerate}
To see $\cP$ exists, consider random equitable partitions $V_1^A, \dots, V_s^A$ of $A$ and $V_1^B, \dots, V_s^B$ of $B$ and let $V_i:=V_i^A\cup V_i^B$. Lemma~\ref{lem:chernoff} implies that this partition satisfies \eqref{rand2*:bip} with probability at least $3/4$.%
\COMMENT{Consider any $x\in V(G)$, wlog $x\in B$. Note that $d_G(x, V_i\cap A)=|V_i^A\cap N_G(x)|$ has hypergeometric distribution with parameters $|A|$, $|V_i^A|$ and $d_G(x)$.
By Lemma~\ref{lem:chernoff}, $$\pr(d_G(x, V_i)< (1/2+2\eps/3)|V_i^A|)\leq 2e^{-2(\eps n/4s)^2/(n/s)}\leq 2e^{-\eps^2n/8s}\leq 1/n^2.$$}

Since $|V_i|\leq \eps n/2$ for all $1\leq i\leq s$, it suffices to show that $G[\cP]$ has an approximate $C_{2k}$-decomposition $\cF$ such that $\Delta(G[\cP]-\bigcup \cF)\leq \eps n/2$ .
Let $\{T_1,\dots,T_\ell\}$ be a $K_t$-decomposition of $K_s$, where $V(K_s)=\{1, \dots,s\}$. For each $1\leq i\leq \ell$, define $G_i:=\bigcup_{jk \in E(T_i)}G[V_j,V_k]$, so the $G_i$ decompose $G[\cP]$.
For each $1\leq i\leq \ell$, each $X\in \{A,B\}$ and all $x\in V(G_i)\setminus X$, we have
\begin{align*}
d_{G_i}(x) \geq (t-1)(1/2+2\eps/3)\lfloor |X|/s\rfloor \geq (1/2+\eps/2)t\lceil |X|/s\rceil\geq (1/2+\eps/2)|V(G_i)\cap X|,
\end{align*}
by \eqref{rand1*:bip} and \eqref{rand2*:bip}. So $\dbip(G_i)\geq 1/2+\eps/2$. We also note that 
$$n':=t\lfloor n/s\rfloor \leq |V(G_i)\cap A|, |V(G_i)\cap B|\leq t\lceil 3n/s\rceil\leq 4n'.$$

Let $X_1:=\emptyset$. For each $1\leq i \leq \ell$ in turn, apply Lemma~\ref{lem:careofbad:bip} (with $G_i$, $\eps/2$ and $X_i\cap V(G_i)$ playing the roles of $G$, $\eps$ and $X$) to find $H_i\subseteq G_i$ such that $G_i-H_i$ is $C_{2k}$-decomposable, $d_{H_i}(x) \le \eta n'$ for all $x\in X_i$ and $|Y_i| \leq \eta n'$, where $Y_i:=\{x\in V(G_i): d_{H_i}(x)> \eta n'\}$. Let $X_{i+1}:=X_i\cup Y_i$. Note that, for all $1\leq i\leq \ell$,
$|X_i| \le s^2\eta n'\le \eta^{1/2}n'$ so we can indeed use Lemma~\ref{lem:careofbad:bip}. Let $H:=\bigcup_{i=1}^{\ell}H_i$ and consider any $x\in V(G)$. We know that
$$d_H(x) \le \ell\eta n' + 4n' \le (s^2\eta+4)tn/s\le \eps n/2,$$
since
$d_{H_i}(x)\leq\eta n'$ for all but at most one $1\leq i\leq \ell$.
\end{proof}

The following proposition takes a subset $R$ of $V(G)$ and covers all the edges in a sparse subgraph $H$ of $G[\overline R]$ using copies of $C_{2k}$ without using any vertex too many times. It is an analogue of Proposition~5.10 in~\cite{stef} and the proof is identical, so we omit the details.%
\COMMENT{\begin{proof}
Let $e_1,\dots,e_m$ be an enumeration of $E(H)$. We aim to find edge-disjoint copies $F_1,\dots,F_m$ of $C_{2k}$ in $G$ such that $F_i$ contains $e_i$ and $V(F_i)\cap L =V(e_i)$. Suppose we have already found $F_1,\dots,F_{j-1}$ for some $1\leq j\leq m$. Let $G_{j-1}:=\bigcup_{i=1}^{j-1} F_i$ and suppose that $\Delta(G_{j-1}) \le \sqrt{\gamma}n+2$. Let $X:=\{x\in V(G):d_{G_{j-1}}(x)>\sqrt{\gamma} n\}$. Note that $X\cap L=\emptyset$, since $d_{G_{j-1}}(x)\le 2\Delta(H) \le \sqrt{\gamma} n$ for all $x\in L$.%
\COMMENT{Since $V(F_i)\cap L =V(e_i)$, $x$ is only contained in $F_i$'s when $e_i$ is incident to $x$}
We have $$|X|\sqrt\gamma n\leq 2e(G_{j-1}) \le 4ke(H) \le 20k\gamma n^2,$$
giving $|X| \le 20k\sqrt{\gamma}n \le \gamma^{1/3}\mu n$. Let $G':=(G-G_{j-1})[(R\setminus X) \cup V(e_j)]$.
Since $\dbip\geq 1/2+\mu/2$, we can find a copy $F_j$ of $C_{2k}$ in $G'$ that contains $e_j$. Moreover, $F_j$ avoids $X$ so $\Delta(G_j)\le \sqrt{\gamma}n+2$. Letting $J:=\bigcup_{i=1}^{m}(F_i-e_i)$ completes the proof.
\end{proof}}

\begin{prop} \label{prop:cover-sparse-graph:bip}
Let $k\in \N$, $k\geq 2$ and $1/n \ll \gamma \ll \mu, 1/k$. Let $G=(A,B)$ be a bipartite graph with $n\leq |A|,|B|\leq 5n$. Let $V(G)=L\cupdot R$ such that $|R\cap X|\ge \mu n$ and $d_G(x, R\cap X)\geq (1/2+\mu)|R\cap X|$ for each $X\in \{A,B\}$ and all $x\in V(G)\setminus X$.
Let $H$ be any subgraph of $G[L]$ such that $\Delta(H) \le \gamma n$. Then there exists $J\subseteq G$ such that $J[L]$ is empty, $J\cup H$ is $C_{2k}$-decomposable and $\Delta(J)\le \mu^2 n$.
\end{prop}

We now use each of the results obtained so far to prove Lemma~\ref{lem:coverdown:bip}. This lemma forms the basis of the induction proof of Lemma~\ref{lem:nearoptimal:bip}.

\begin{lemma}\label{lem:coverdown:bip}
Let $k\in \N$, $k\geq 2$ and $1/n\ll \mu\ll 1/k$. Let $G=(A,B)$ be a bipartite graph with $n\leq |A|,|B|\leq 3n$. Let $U\subseteq V(G)$ with $|U\cap A|=\lfloor \mu |A|\rfloor$ and $|U\cap B|=\lfloor \mu |B|\rfloor$. Suppose $\dbip(G)\geq 1/2+2\mu$ and $d_G(x, U\cap X)\geq (1/2+\mu)|U\cap X|$ for each $X\in \{A,B\}$ and all $x\in V(G)\setminus X$. Then, if $2\mid d_G(x)$ for all $x\in V(G)\setminus U$, there exists a collection $\cF$ of edge-disjoint copies of $C_{2k}$ such that every edge in $G-G[U]$ is covered and $\Delta (\bigcup \cF[U])\leq \mu^3 |U|$.
\end{lemma}

\begin{proof}
Choose constants $\gamma, \xi$ such that
$1/n\ll\gamma\ll \xi \ll\mu\ll 1/k$.
Let $W:=V(G)\setminus U$, $m:=\lceil \xi^{-1} \rceil$ and $M:=\binom{m+1}{2}$. 
Let $V_1,\dots,V_M$ be a partition of $U$ such that for all $1\leq i\leq M$, each $X\in \{A,B\}$ and all $x\in V(G)\setminus X$:
\begin{enumerate}
\item $d_G(x, V_i\cap X)\geq (1/2+\mu/2)|V_i\cap X|$;\label{coverdown1:bip}
\item $|V_i\cap X|=\lfloor |U\cap X|/M \rfloor$ or $\lceil |U\cap X|/M\rceil$.%
\COMMENT{so $\lfloor \mu n\rfloor/M-1\leq |V_i\cap A|, |V_i\cap B|\leq \lfloor \mu 3n\rfloor/M+1\leq 4(\lfloor \mu n\rfloor/M-1)$}
\end{enumerate}
To see that such a partition exists, consider random equipartitions $V_1^A, \dots, V_M^A$ of $U\cap A$ and $V_1^B, \dots, V_M^B$ of $U\cap B$. Let $V_i:=V_i^A\cap V_i^B$. Lemma~\ref{lem:chernoff} implies that this partition satisfies \eqref{coverdown1:bip} with probability at least $3/4$.%
\COMMENT{Note $M=\binom{m+1}{2}\leq (1/\xi+2)^2/2\leq 5/\xi^2$.
Consider any $x\in V(G)$, wlog $x\in B$. $d_G(x, V_i)$ has hypergeometric distribution with parameters: $|U\cap A|$, $|V_i^A|$ and $d_G(x, U)$. 
Lemma~\ref{lem:chernoff} gives
%$\xi^2\mu^3n/25\geq 4n\log n$.
$$\pr(d_G(x, V_i)< (1/2+\mu/2)|V_i^A|)\le 2 e^{-2\mu^2|U\cap A|/9M}\leq 2e^{-\xi^2\mu^3n/25}\leq 1/n^2.$$ 
By summing over all choices of $i$ and $x$, we see that with probability at least $1-6M/n\geq 3/4$ the partition chosen in this way works.}

Let $W_1, \dots, W_m$ be a partition of $W$ such that $W_1\cap A, \dots, W_m\cap A$ and $W_1\cap B, \dots, W_m\cap B$ are equipartitions of $W\cap A$ and $W\cap B$ respectively.%
\COMMENT{so $(1-\mu)n/M-1\leq |W_i\cap A|, |W_i\cap B|\leq 3n/M\leq 4((1-\mu)n/M-1)$}
Let $G_W^1, \dots, G_W^M$ be an enumeration of the $M$ graphs of the form $G[W_i]$ or $G[W_i,W_j]$. Note $G[W]=\bigcup_{i=1}^MG_W^i$ and, for all $1\leq i\leq M$, 
\begin{equation}\label{eq:spargraph1}
|V(G_W^i)\cap A|, |V(G_W^i)\cap B|\leq 2(3n/m+1)\leq 7\xi n
\end{equation}
For each $1\leq i\leq M$, let $R_i:=G[V_i, V(G_W^i)]$. Let $R:=\bigcup_{i=1}^MR_i$. For each $v\in V_i$ we see that $d_R(v)\leq 7\xi n$ by \eqref{eq:spargraph1} and for each $v\in W$, we have $d_R(v)\leq m((3n\mu/M)+1)\leq 7\xi n$.%
\COMMENT{$m((3n\mu/M)+1)\leq m(7n/2M)=7n/(m+1)\leq 7\xi n$}
Thus $\Delta(R)\leq 7\xi n$. 

Let $G':=G-(G[U]\cup R)$. Since $|U\cap A|=\lfloor \mu |A|\rfloor$, $|U\cap B|=\lfloor \mu |B|\rfloor$ and $\Delta(R)\leq 7\xi n$, we note that $\dbip(G')\geq 1/2+\mu/2$. So, by Lemma~\ref{lem:boundmaxdegree:bip} (with $\gamma$ playing the role of $\eps$), $G'$ has an approximate $C_{2k}$-decomposition $\cF_1$ such that $H:=G'-\bigcup \cF_1$ satisfies $\Delta(H)\leq \gamma n$.

We now use $R$ and Proposition~\ref{prop:cover-sparse-graph:bip} to cover the edges in $H[W]$. For each $1\leq i\leq M$, let $H_i:=H[W]\cap G_W^i$ (so $H[W]=\bigcup H_i$) and $G_i:=G[V_i]\cup R_i\cup H_i$. 
Observe that $G_i$ is a bipartite graph and $V(G_i)=V_i\cup V(G_W^i)$. Let us check that $G_i$ satisfies the conditions of Proposition~\ref{prop:cover-sparse-graph:bip} (with $G_i$, $\sqrt\gamma$, $\xi^2$ and $V_i$ playing the roles of $G$, $\gamma$, $\mu$ and $R$).
Let $n_i:=\min\{|V(G_i)\cap A|, |V(G_i)\cap B|\}$, then
$$n_i\leq |V(G_i)\cap A|, |V(G_i)\cap B|\leq 4n_i.$$
Note that
\begin{align*}
n_i\leq |V(G_i)\cap A|= |V_i\cap A|+|V(G_W^i)\cap A|\stackrel{\eqref{eq:spargraph1}}{\leq} 3\mu n/M+7\xi n\leq 8\xi n
\end{align*}
which gives $n\geq n_i/8\xi$. We use this to see that
$$|V_i\cap A|,|V_i\cap B| \geq \mu n/2M \geq \mu\xi^2 n/2 \geq \xi^2 n_i.$$
Also $\Delta(H_i)\leq \gamma n\leq \sqrt\gamma n_i$ and \eqref{coverdown1:bip} implies that 
$d_{G_i}(x, V_i\cap X)\geq (1/2+\xi^2)|V_i\cap X|$ for each $X\in \{A,B\}$ and all $x\in V(G_i)\setminus X$. So we may apply Proposition~\ref{prop:cover-sparse-graph:bip} to find $J_i\subseteq G_i$ such that $J_i[V(G_i)\setminus V_i]$ is empty, $J_i\cup H_i$ is $C_{2k}$-decomposable and $\Delta(J_i)\le \xi^4 n_i$. Let $J:=\bigcup_{i=1}^M J_i$. Then $J\cup H[W]$ has a $C_{2k}$-decomposition $\cF_2$ and $\Delta(J)\leq \xi n$.

We must now cover the remaining edges in $H[U,W]\cup R$. Let $G'':=G-\bigcup (\cF_1\cup \cF_2)$. Note that $G''[W]$ is empty and
$$\Delta(G'')\leq \Delta(H)+\Delta(R)\leq \gamma n+7\xi n\leq 8\xi n.$$
Since $\Delta(J)\leq \xi n$, $\dbip(G''[U])\geq 1/2+\mu/2$. For each $w\in W$, $d_{G''}(w)$ is even, so we can pair up the vertices in $N_{G''}(w)$ arbitrarily and let $P$ denote the list of pairs of all neighbours of $W$. Each vertex in $U$ appears in at most $\Delta(G'')\leq \sqrt\xi |U|$ of the pairs in $P$ and 
$|P|\leq \Delta(G'')3n\leq \sqrt\xi |U|^2$. The vertices in each pair lie in the same vertex class so we can find paths of (even) length $2k-2$ between each pair so that these paths are edge-disjoint and no vertex is used more than $\mu^3|U|/4$ times.%
\COMMENT{We need to find $\leq \sqrt\xi |U|^2$ paths. Let $B$ be the set of vertices used at least $\mu^4|U|$ times as an internal vertex on one of these paths. At any stage $|B|\leq 2k\sqrt\xi |U|^2/\mu^4|U|\leq \mu|U|/20$. Since $\dbip(G''[U])\geq 1/2+\mu/2$ we are always able to choose the interior vertices to avoid $B$. Thus getting $\Delta(\bigcup\cF_3)\leq \sqrt\xi |U|+2(\mu^4|U|+2)\leq \mu^3 |U|/2$.}
We obtain a collection $\cF_3$ of edge-disjoint copies of $C_{2k}$ which cover the edges of $G''-G''[W]$ such that $\Delta(\bigcup\cF_3)\leq \mu^3 |U|/2$.
Let $\cF:=\cF_1\cup \cF_2\cup\cF_3$. Then
$$\Delta \big(\bigcup\cF[U]\big)\leq \Delta(J)+\Delta\big(\bigcup\cF_3\big) \leq \mu^3|U|$$
and $\cF$ covers every edge of $G-G[U]$.
\end{proof}

Finally, we use Lemma~\ref{lem:coverdown:bip} and induction to prove Lemma~\ref{lem:nearoptimal:bip}.

\begin{proofof}{Lemma~\ref{lem:nearoptimal:bip}}
If $\ell=0$, we can set $H_\ell:=G$, so we assume $\ell\geq 1$. We begin by observing that for any $0\leq i\leq \ell$, we have $2\mu^i n/3\leq \mu^i n-1/(1-\mu)\leq |U_i\cap A|, |U_i\cap B|\le 2\mu^i n$. 
The lemma will follow from the following statement which we will prove by induction on $\ell$.

\begin{quote}
\textit{Let $G=(A,B)$ be a $2$-divisible bipartite graph with $\dbip\geq 1/2+3\mu$ and $|A|\leq |B|\leq 3|A|$. Let $U_1\subseteq V(G)$ with $|U_1\cap A|=\lfloor \mu |A|\rfloor$ and $|U_1\cap B|=\lfloor \mu |B|\rfloor$. Suppose that $d_G(x, U_1\cap X)\geq (1/2+7\mu/2)|U_1\cap X|$ for each $X\in \{A,B\}$ and all $x\in V(G)\setminus X$. Let $U_1 \supseteq \dots \supseteq U_\ell$ be a $(1/2+4\mu,\mu, m)$-vortex respecting $(U_1\cap A, U_1\cap B)$ in $G[U_1]$ such that $|U_i\cap B|\leq 3|U_i\cap A|$, for each $1\leq i\leq \ell$. Then there exists $H_\ell\subseteq G[U_\ell]$ such that $G-H_\ell$ is $C_{2k}$-decomposable.}
\end{quote}

If $\ell=1$, the statement follows directly from Lemma~\ref{lem:coverdown:bip} applied to $G$ and $U_1$. Assume then that $\ell\geq 2$ and the statement holds for $\ell-1$. Let $G':=G-G[U_2]$ and note that $\dbip(G')\geq 1/2+2\mu$
and $d_{G'}(x,U_1\cap X)\geq (1/2+\mu)|U_1\cap X|$ for each $X\in \{A,B\}$ and all $x\in V(G)\setminus X$.
Furthermore, for all $x\in V(G')\setminus U_1$, $d_{G'}(x)=d_G(x)$ so $2\mid d_{G'}(x)$. Apply Lemma~\ref{lem:coverdown:bip} to find an edge-disjoint collection $\cF$ of copies of $C_{2k}$ covering all edges in $G'-G[U_1]$ such that
$$\Delta\big(\bigcup \cF[U_1]\big)\leq \mu^3|U_1|\leq 5\mu^2|U_2\cap A|.$$
Let $G'':=G[U_1]-\bigcup \cF$. Then $G''$ is a $2$-divisible bipartite graph with $\dbip(G'')\geq 1/2+3\mu$. For each $X\in \{A,B\}$, $|U_2\cap X|=\lfloor \mu |U_1\cap X|\rfloor$ and, for any $x\in V(G'')\setminus X$,
$$d_{G''}(x,U_2\cap X)\geq (1/2+4\mu)|U_2\cap X|-\Delta\big(\bigcup \cF[U_1]\big)\geq (1/2+7\mu/2)|U_2\cap X|.$$
Since $G''[U_2]=G[U_2]$, $U_2 \supseteq \dots \supseteq U_\ell$ is a $(1/2+4\mu,\mu,m)$-vortex respecting $(U_2\cap A, U_2\cap B)$ in $G''[U_2]$. Hence, by induction, there exists a subgraph $H_\ell$ of $G[U_\ell]$ such that $G''-H_\ell$ has a  $C_{2k}$-decomposition $\cF'$. Together $\cF\cup \cF'$ is a $C_{2k}$-decomposition of $G-H_\ell$.
\end{proofof}

\section{Decompositions of expanders}\label{sec:expanderdecomp}
The purpose of this section is to prove Theorem~\ref{thm:nuexpander}. This result finds a $C_{2k}$-decomposition of any $C_{2k}$-divisible $\nu$-expander $G$ when $k\geq 4$.%
\COMMENT{cannot find the absorber for $C_6$.}
The significance of $G$ being a $\nu$-expander (defined in Section~\ref{sec:longercycle}) is that there are many internally disjoint paths between any pair of vertices in $G$. We can use these paths to construct copies of $C_{2k}$ and to find absorbers and this allows us to use the arguments of \cite{stef} with only slight modification. We will make use of the fact that $\nu$-expansion is a robust property in the sense that it is not destroyed when we remove a sparse subgraph.

\subsection{Finding paths}

The next result can be used to find many internally disjoint paths with predetermined endpoints without using any vertex too often.%
\COMMENT{if we ask for minimum degree $n/2$, we can have paths of length three here}

\begin{prop}\label{prop:pairing}
Let $k\in \N$, $k\geq 4$ and $1/n\ll \gamma\ll \nu, 1/k$. Let $G$ be a $\nu$-expander on $n$ vertices and let $P=\{(x_1, y_1), \dots, (x_m, y_m)\}$ be a collection of $m\leq \gamma n^2$ pairs of distinct vertices of $G$. Suppose that each vertex appears in at most $\gamma n$ pairs in $P$. Then $G$ contains a collection of edge-disjoint paths $\cP=\{P^1, \dots, P^m\}$ such that, for each $1\leq i\leq m$, $P^i$ is a path of length $k$ from $x_i$ to $y_i$. Furthermore, $\Delta(\bigcup \cP)\leq \gamma^{1/3} n$.
\end{prop}

\begin{proof}
Let $1\leq j\leq m$ and suppose we have already found paths $P^1,\dots, P^{j-1}$ such that each vertex in $V(G)$ appears as an internal vertex in at most $2\sqrt \gamma n$ of the paths. Let $B$ be the set of all vertices which appear as an internal vertex in at least $\sqrt\gamma n$ paths in $P^1, \dots, P^{j-1}$. Note that 
$$|B|\leq m(k-1)/(\sqrt\gamma n)\leq \nu^2n.$$
Let $G_j:=G-\bigcup_{i=1}^{j-1}P^i$. Note that $\Delta(\bigcup_{i=1}^{j-1}P^i)\leq 4\sqrt\gamma n+\gamma n$ so $G_j$ is a $\nu/2$-expander (which implies $\delta(G_j)\geq \nu n/2$).
We find a path $P^j$ between $x_j$ and $y_j$ in $G_j$ whose interior vertices avoid $B$ as follows. Since $\nu n/2\geq |B|+k$, we can embed a path of length $k-4$ starting at $x_j$ greedily. Let $x_j'$ denote its endpoint. In order to find a path of length four between $x_j'$ and $y_j$ it suffices to note that
$$|R_{\nu/2, G_j}(N_{G_j}(x_j'))\cap R_{\nu/2, G_j}(N_{G_j}(y_j))|\geq \nu n\geq |B|+k.$$
Continuing in this way, we obtain edge-disjoint paths $P^1, \dots, P^m$ of length $k$ such that no vertex is used as an internal vertex more than $2\sqrt\gamma n$ times. Thus $\Delta(\bigcup \cP)\leq 4\sqrt\gamma n+\gamma n\leq \gamma^{1/3} n$.
\end{proof}

\subsection{Expander vortices}
We now introduce a further variant of the vortex, this time for expanders, where we replace the minimum degree condition with an expansion property instead. Let $G$ be a graph on $n$ vertices. A \emph{$(\nu, \mu,m)$-expander vortex} in $G$ is a sequence $U_0 \supseteq U_1 \supseteq \dots \supseteq U_\ell$ such that
\begin{itemize}
\item $U_0=V(G)$;
\item $|U_i|=\lfloor \mu|U_{i-1}| \rfloor$, for all $1\leq i\leq \ell$, and $|U_\ell|=m$;
\item $N_G(x, U_i)$ is $\nu$-expanding in $G[U_i]$, for all $1\leq i\leq \ell$ and all $x\in U_{i-1}$.
\end{itemize}

\begin{prop} \label{prop:randomsubset}
Let $0\leq \nu\leq 1$ and $1/n\ll \mu<1$. Suppose that $G$ is a $\nu$-expander on $n$ vertices. Then there exists $U\subseteq V(G)$ of size $\lfloor \mu n \rfloor$ such that, for every $x\in V(G)$, $N_G(x,U)$ is $(\nu-n^{-1/3})$-expanding in $G[U]$.
\end{prop}

\begin{proof}
Let $U$ be a random subset of $V(G)$ of size $\lfloor \mu n \rfloor$. Fix $x\in V(G)$. Lemma~\ref{lem:chernoff} gives
\begin{align*}
\pr(|R_{\nu,G}(N_G(x))\cap U|<(1/2+\nu-n^{-1/3})|U|) &\le 2 e^{-2 n^{-2/3}|U|^2/n}\leq 2e^{-\mu^2n^{1/3}}\leq 1/n^3.
\end{align*}
Consider any $y\in R_{\nu,G}(N_G(x))$. Again by Lemma~\ref{lem:chernoff},
$$\pr(d_G(y,N_G(x,U))<(\nu-n^{-1/3}) |U|) \le 2 e^{-2n^{-2/3}|U|^2/n}\leq 2e^{-\mu^2n^{1/3}}\leq 1/n^3.$$%
\COMMENT{$=(N_G(y)\cap N_G(x))\cap U$}

By summing over all choices of $x$ and $y$, we see that with probability at least $1-2/n$ the set $U$ chosen in this way satisfies:
\begin{enumerate}
\item $|R_{\nu,G}(N_G(x))\cap U|\geq (1/2+\nu-n^{-1/3})|U|$, for all $x\in V(G)$ and\label{item:rand2}
\item $d_G(y,N_G(x,U))\ge (\nu-n^{-1/3})|U|$, for all $x\in V(G)$ and all $y\in R_{\nu,G}(N_G(x))$.\label{item:rand3}
\end{enumerate}
For any $x\in V(G)$, we have
$$|R_{\nu-n^{-1/3},G[U]}(N_G(x,U))|\stackrel{\eqref{item:rand3}}{\ge}|R_{\nu,G}(N_G(x))\cap U| \stackrel{\eqref{item:rand2}}{\geq} (1/2+\nu-n^{-1/3})|U|$$
so $U$ is the required set.
\end{proof}

We use the following result to find an expander vortex in $G$.

\begin{lemma} \label{lem:getvortex}
Let $0\leq \nu \leq 1$ and $1/m'\ll \mu<1$. Suppose that $G$ is a $\nu$-expander on $n \ge m'$ vertices. Then $G$ has a $(\nu-\mu, \mu, m)$-expander vortex for some $\lfloor \mu m' \rfloor \le m \le m'$.
\end{lemma}

\noindent This result follows from repeated applications of Proposition~\ref{prop:randomsubset} (see the proof of Lemma~4.3 in~\cite{stef} for more details).%
\COMMENT{See appendix~\ref{append:expander}}

\subsection{Covering most of the edges}

In this section we decompose almost all of the graph $G$ into cycles except for a very restricted remainder using the following result. This is exactly the technique we used in Section~\ref{sec:bipartite}, so again we omit some details.

\begin{lemma} \label{lem:nearoptimal}
Let $k\in \N$, $k\geq 3$ and $1/m \ll \nu,1/k$. Let $G$ be a $2$-divisible $4\nu$-expander and let $U_0 \supseteq U_1 \supseteq \dots \supseteq U_\ell$ be a $(5\nu,\nu,m)$-expander vortex in $G$. Then there exists $H_\ell\subseteq G[U_\ell]$ such that $G-H_\ell$ is $C_{2k}$-decomposable.
\end{lemma}

We require some preliminary results. The first finds an approximate $C_{2k}$-decomposition of $G$ whilst maintaining control over the number of edges incident at all vertices in a given set $X$.

\begin{lemma} \label{lem:careofbad}
Let $k\in \N$, $k\geq 3$ and $1/n \ll \eta \ll \nu,1/k$. Suppose that $G$ is a $\nu$-expander on $n$ vertices and that $X\subseteq V(G)$ of size at most $\eta^{1/2} n$. Then there exists $H\subseteq G$ such that $G-H$ is $C_{2k}$-decomposable, $Y:=\{x\in V(G):d_H(x)>\eta n\}$ has size at most $\eta n$ and $X\cap Y=\emptyset$.
\end{lemma}

\begin{proof}
We begin by finding edge-disjoint copies of $C_{2k}$ which cover all the edges in $G[X]$. To this end, let $P_X:=\{(x,y): xy\in E_G(X)\}$. Since $|X|\leq \eta^{1/2}n$, $G-G[X]$ is a $3\nu/4$-expander and we may apply Proposition~\ref{prop:pairing} (with $P_X$, $\eta^{1/2}$, $G-G[X]$, $2k-1$ and $3\nu/4$ playing the roles of $P$, $\gamma$, $G$, $k$ and $\nu$) to find a collection $\cP_X$ of edge-disjoint paths of length $2k-1$ between the endpoints of each edge in $E_G(X)$ such that $\Delta(\bigcup \cP_X)\leq \eta^{1/6}n$. Thus we obtain a collection $\cF_X$ of edge-disjoint copies of $C_{2k}$ which cover all of the edges in $G[X]$ such that $\Delta(\bigcup \cF_X)\leq 2\eta^{1/6}n$. Let $G':=G\setminus \bigcup \cF_X$.

Our next step is to cover all but at most one of the remaining edges incident at each vertex in $X$. For each $x\in X$, pair up the vertices in $N_{G'}(x)$, leaving at most one vertex. Let $P_X'$ denote the list of pairs for all $x\in X$. Note that $G'\setminus X$ is a $\nu/2$-expander. Then, as previously, apply Proposition~\ref{prop:pairing} (with $P_X'$, $\eta^{1/2}$, $G'\setminus X$, $2k-2$ and $\nu/2$ playing the roles of $P$, $\gamma$, $G$, $k$ and $\nu$) to find a collection $\cP_X'$ of edge-disjoint paths of length $2k-2$ in $G'\setminus X$ between each pair in $P_X'$. These paths combine with edges incident at $X$ to form a collection $\cF_X'$ of edge-disjoint copies of $C_{2k}$ which, together with $\cF_X$, cover all but at most one edge incident at each $x\in X$.

Finally, let $H':=G-\bigcup (\cF_X\cup \cF_X')$. Use the Erd\H{o}s-Stone theorem to greedily find an $\eta^3$-approximate $C_{2k}$-decomposition of $H'$ which we will denote by $\cF$. Let $H:=H'-\bigcup \cF$ and note that $G-H$ has a $C_{2k}$-decomposition given by $\cF_X\cup \cF_X'\cup \cF$. If $Y:=\{x\in V(G):d_H(x)>\eta n\}$, then $|Y|\leq 2e(H)/(\eta n)\leq \eta n$. Since $d_H(x)\leq 1$ for all $x\in X$, $X\cap Y=\emptyset$.
\end{proof}

We use Lemma~\ref{lem:careofbad} to prove the following result which finds a $C_{2k}$-decomposition of $G$ so that every vertex has low degree in the remainder.

\begin{lemma} \label{lem:boundmaxdegree}
Let $k\in \N$, $k\geq 3$ and $1/n \ll \nu,1/k$.
Let $G$ be a $\nu$-expander on $n$ vertices. Then $G$ has an approximate $C_{2k}$-decomposition $\cF$ such that $\Delta(G-\bigcup \cF)\leq \nu n$.
\end{lemma}

\begin{proof}
Choose $s, t\in \N$ and $\eta>0$ such that 
$$1/n\ll \eta \ll 1/s  \ll 1/t \ll \nu,1/k$$
and $K_s$ has a $K_t$-decomposition ($s$ and $t$ exist by Proposition~\ref{prop:clique dec}). 
Let $\cP=\{V_1,\dots,V_s\}$ be an equipartition of $V(G)$ satisfying the following for all $1\leq i\leq s$:
\begin{enumerate}[\rm(i)]
\item $d_G(y,N_G(x,V_i)) \ge (\nu-\eta)|V_i|$ for all $x\in V(G)$ and all $y\in R_{\nu,G}(N_G(x))$;\label{rand2*}
\item $|R_{\nu,G}(N_G(x))\cap V_i|\geq (1/2+\nu-\eta)|V_i|$ for all $x\in V(G)$.\label{rand3*}
\end{enumerate}
To see that such a partition exists, consider a random equipartition of $V(G)$ into $s$ parts and apply Lemma~\ref{lem:chernoff} to see that this partition satisfies \eqref{rand2*}--\eqref{rand3*} with probability at least $3/4$.%
\COMMENT{$d_G(y,N_G(x,V_i))$ has hypergeometric distrbution with parameters: $n$, $|V_i|$ and $N_G(x)\cap N_G(y)$. $|R_{\nu,G}(N_G(x))\cap V_i|$ has hypergeometric distribution with parameters: $n$, $|V_i|$ and $|R_{\nu,G}(N_G(x))|$. By Lemma~\ref{lem:chernoff}: $\pr(d_G(y,N_G(x,V_i))<(\nu-\eta)|V_i|), \pr(|R_{\nu,G}(N_G(x))\cap V_i|<(1/2+\nu-\eta)|V_i|)\leq 2e^{-\eta^2n/s}\leq 1/n^3$. Sum over all $x,y,i$, $\leq 2n^2s/n^3\leq 1/4$.}
It will suffice to show that $G[\cP]$ has an approximate $C_{2k}$-decomposition $\cF$ such that $\Delta(G[\cP]-\bigcup \cF)\leq \nu n/2$ (since $|V_i|\leq \nu n/2$ for all $1\leq i\leq s$).

Consider $\{T_1,\dots,T_\ell\}$, a $K_t$-decomposition of $K_s$, where $V(K_s)=\{1,\dots, s\}$. For each $1\leq i\leq \ell$, define $G_i:=\bigcup_{jk \in E(T_i)}G[V_j,V_k]$, so the $G_i$ decompose $G[\cP]$.
Consider any $x\in V(G_i)$ and any $y\in R_{\nu,G}(N_G(x))\cap V(G_i)$. We have
\begin{align}\label{eq:boundmaxdeg}
d_{G_i}(y, N_{G_i}(x))&=\sum_{\substack{V_j\subseteq V(G_i)\\x,y\not\in V_j}}d_G(y, N_G(x, V_j))\stackrel{\mathclap{\eqref{rand2*}}}{\geq} (t-2)(v-\eta)\lfloor n/s\rfloor\\
&\geq (\nu/2)t\lceil n/s\rceil\geq (\nu/2)|G_i|.\nonumber
\end{align}
So
\begin{align*}
|R_{\nu/2,G_i}(N_{G_i}(x))|&\stackrel{\mathclap{\eqref{eq:boundmaxdeg}}}{\geq} |R_{\nu,G}(N_G(x))\cap V(G_i)| \stackrel{\mathclap{\eqref{rand3*}}}{\geq} t(1/2+\nu-\eta)\lfloor n/s\rfloor \\
&\geq (1/2+\nu/2)t\lceil n/s\rceil \geq (1/2+\nu/2)|G_i|.
\end{align*}
Thus $G_i$ is a $\nu/2$-expander for each $1\leq i\leq \ell$.

Let $X_1:=\emptyset$. For each $1\leq i \leq \ell$ in turn, apply Lemma~\ref{lem:careofbad} (with $G_i$, $\nu/2$ and $X_i\cap V(G_i)$ playing the roles of $G$, $\nu$ and $X$) to find $H_i\subseteq G_i$ such that $G_i-H_i$ is $C_{2k}$-decomposable, $d_{H_i}(x) \le \eta |G_i|$ for all $x\in X_i$ and $|Y_i| \leq \eta |G_i|$, where $Y_i:=\{x\in V(G_i): d_{H_i}(x)> \eta |G_i|\}$.
Let $X_{i+1}:=X_i\cup Y_i$. Note that, for all $1\leq i\leq \ell$,
$|X_i| \le s^2 \eta t\lceil n/s\rceil \le \eta^{1/2} t\lfloor n/s\rfloor$, so we can indeed use Lemma~\ref{lem:careofbad}. Let $H:=\bigcup_{i=1}^{\ell}H_i$ and consider any $x\in V(G)$. We know that
$$d_H(x) \le \ell \eta t\lceil n/s\rceil + t\lceil n/s\rceil \le 2s\eta tn + 2tn/s \le \nu n/2,$$
since
$d_{H_i}(x)\leq\eta t\lceil n/s\rceil $ for all but at most one $1\leq i\leq \ell$.
\end{proof}

The following proposition, an analogue of Proposition~\ref{prop:cover-sparse-graph:bip}, takes a subset $R$ of $G$ and covers all the edges in a sparse subgraph $H$ which have no endpoint in this set $R$. It is proved by mimicking the proof of Proposition~5.10 in~\cite{stef} (see \cite{thesis} for more details).%
\COMMENT{see Appendix~\ref{append:expander}}

\begin{prop} \label{prop:cover-sparse-graph}
Let $k\in \N$, $k\geq 3$ and $1/n \ll \gamma \ll \mu, 1/k$. Let $G$ be a graph on $n$ vertices and let $V(G)=L\cupdot R$ such that $|R| \ge \mu n$ and $N_G(x, R)$ is $\mu$-expanding in $G[R]$ for all $x\in V(G)$.
%\COMMENT{CHECK that this expansion condition isn't too strict. Can have separate $\mu$ and $\nu$ if want. Is $\Delta(A)$ small enough?}
Let $H$ be any subgraph of $G[L]$ such that $\Delta(H) \le \gamma n$. Then there exists a subgraph $A$ of $G$ such that $A[L]$ is empty, $A\cup H$ is $C_{2k}$-decomposable and $\Delta(A)\le \gamma^{1/3} |R|$.
\end{prop}

Lemma~\ref{lem:nearoptimal} will follow from the following result by induction. The proof of Lemma~\ref{lem:coverdown} very closely resembles that of Lemma~\ref{lem:coverdown:bip} (and uses Lemma~\ref{lem:boundmaxdegree}, Proposition~\ref{prop:cover-sparse-graph} and Proposition~\ref{prop:pairing}, in this order). We omit the details and refer the reader instead to \cite{thesis}.%
\COMMENT{see Appendix~\ref{append:expander}}

\begin{lemma}\label{lem:coverdown}
Let $k\in \N$, $k\geq 3$ and $1/n\ll \nu,1/k$. Let $G$ be a $3\nu$-expander on $n$ vertices and $U\subseteq V(G)$ with $|U|=\lfloor \nu n\rfloor$. Suppose that $N_G(x,U)$ is $\nu$-expanding in $G[U]$ for all $x\in V(G)$. Then, if $2\mid d_G(x)$ for all $x\in V(G)\setminus U$, there exists a collection $\cF$ of edge-disjoint copies of $C_{2k}$ such that every edge in $G-G[U]$ is covered and $\Delta (\bigcup \cF[U])\leq \nu^2 |U|/4$.
\end{lemma}

Finally, we use Lemma~\ref{lem:coverdown} and induction to prove Lemma~\ref{lem:nearoptimal}.

\begin{proofof}{Lemma~\ref{lem:nearoptimal}}
If $\ell=0$, we can set $H_\ell:=G$, so we assume $\ell\geq 1$. We will prove the following statement (which implies Lemma~\ref{lem:nearoptimal}) by induction on $\ell$. 

\begin{quote}
\textit{Let $G$ be a $2$-divisible $4\nu$-expander and let $U_1\subseteq V(G)$ of size $\lfloor \nu |G|\rfloor$ such that $N_G(x)$ is $9\nu/2$-expanding in $G[U_1]$ for all $x\in V(G)$. Let $U_1 \supseteq \dots \supseteq U_\ell$ be a $(5\nu,\nu,m)$-expander vortex in $G[U_1]$. Then there exists $H_\ell\subseteq G[U_\ell]$ such that $G-H_\ell$ is $C_{2k}$-decomposable.}
\end{quote}

If $\ell=1$, the statement follows directly from Lemma~\ref{lem:coverdown} applied to $G$ and $U_1$. Assume then that $\ell\geq 2$ and the claim holds for $\ell-1$. Let $G':=G-G[U_2]$ and note that $G'$ is a $3\nu$-expander%
\COMMENT{$4\nu n-|U_2|\geq 3\nu n$}
and $N_{G'}(x)$ is $\nu$-expanding%
\COMMENT{$9\nu |U_1|/2-|U_2|\geq \nu|U_1|$}
in $G'[U_1]$ for all $x\in V(G)$.
Furthermore, for all $x\in V(G')\setminus U_1$, $d_{G'}(x)=d_G(x)$ so $2\mid d_{G'}(x)$. Apply Lemma~\ref{lem:coverdown} to find a collection $\cF$ of edge-disjoint copies of $C_{2k}$ covering all edges in $G'-G[U_1]$ such that $\Delta(\bigcup \cF[U_1])\leq \nu^2|U_1|/4$.
Let $G'':=G[U_1]-\bigcup \cF$. Then $G''$ is $2$-divisible and $G''$ is a $4\nu$-expander%
\COMMENT{$9\nu|U_1|/2-2\nu^2|U_1|/4\geq 4\nu|U_1|$.} 
and $U_2\subseteq V(G'')$ with $|U_2|=\lfloor \nu |G''|\rfloor$. Moreover, for any $x\in V(G'')$, $N_{G''}(x)$ is $9\nu/2$-expanding%
\COMMENT{$5\nu|U_2|-\nu^2|U_1|/4\geq 5\nu|U_2|-\nu|U_2|/2$. Only $N_{G''}(x)$ affected, not the robust neighbourhoods.}
in $G[U_2]$. Since $G''[U_2]=G[U_2]$, $U_2 \supseteq \dots \supseteq U_\ell$ is a $(5\nu,\nu,m)$-expander vortex in $G[U_2]$. Hence, by induction, there exists $H_\ell\subseteq G[U_\ell]$ such that $G''-H_\ell$ has a  $C_{2k}$-decomposition $\cF'$. Together $\cF\cup \cF'$ is a $C_{2k}$-decomposition of $G-H_\ell$.
\end{proofof}

Finally, we prove the main result in this section, Theorem~\ref{thm:nuexpander}.

\begin{proofof}{Theorem~\ref{thm:nuexpander}}
Let $m,m'\in \N$ and $\mu$ be such that
$$1/n\ll 1/m'\ll 1/m\ll \mu\ll \nu, 1/k.$$
Let $\nu':=\nu/7$. Apply Lemma~\ref{lem:getvortex} to $G$ to find a $(6\nu', \nu', m)$-expander vortex $U_0\supseteq U_1\supseteq \dots \supseteq U_\ell$ in $G$.
Let $G_1:=G-G[U_1]$. Since $|U_1|\leq \nu'n$, $G_1$ is a $\nu/2$-expander (and if $k=4$, $\delta(G_1)\geq n/2-\nu'n$) which implies that between any two vertices in $G_1$, there are at least $m'$ internally disjoint paths of length $k-1$. Apply Lemma~\ref{lem:abs:c2k} to the graph $G_1$ with $U_\ell$ playing the role of $U$ to find $A^*\subseteq G_1$ as in the lemma. Let $G^*:=G-A^*$ and note that $G^*$ is $C_{2k}$-divisible. We have $\Delta (A^*)\leq |A^*|\leq 2^{m^2}$, so $G^*$ is a $4\nu'$-expander and $U_0\supseteq U_1\supseteq \dots \supseteq U_\ell$ is a $(5\nu', \nu', m)$-vortex in $G^*$. Then apply Lemma~\ref{lem:nearoptimal} to $G^*$ to find $H_\ell\subseteq G^*[U_\ell]$ such that $G^*-H_\ell$ has a $C_{2k}$-decomposition. Observing that $A^*\cup H_\ell$ has a $C_{2k}$-decomposition (by Lemma~\ref{lem:abs:c2k}) completes the proof.
\end{proofof}

\section{Concluding remarks}\label{sec:conclud}

In Theorem~\ref{thm:mainc2k}, we have found exact minimum degree bounds for a graph to have a decomposition into cycles of all even lengths apart from six. For cycles of length six, the best bound is given by Theorem~\ref{thm:witheps} and remains at $(1/2+\eps)|G|$ which is asymptotically best possible. We conjecture that the bound should also be $|G|/2$ in this case but were unable to prove this using the methods of Section~\ref{sec:longercycle}. The primary reason for this was that we were unable to construct a $C_6$-absorber which could be found in a $\nu$-expander. The transformer construction given in Section~\ref{sec:longerabs} works well for longer cycles since these transformers can be constructed using paths of length at least three between the fixed vertices. But, when the cycle is shorter, we do not have enough flexibility when choosing the intermediate vertices. This means that we were only able to prove the expander decomposition result, Theorem~\ref{thm:nuexpander}, for cycles of length at least eight. There are also places in the proofs of Lemmas~\ref{lem:closeclique}~and~\ref{lem:closebip} where we require the cycle to have length at least eight, though it is likely that these arguments could be adapted for $C_6$-decompositions if required.

\section*{Acknowledgements}
The author would like to thank Daniela K\"uhn and Deryk Osthus for their guidance and helpful comments.

\medskip

{\footnotesize \obeylines \parindent=0pt

Amelia Taylor 
School of Mathematics
University of Birmingham
Edgbaston
Birmingham
B15 2TT
UK
}
\begin{flushleft}
{\it{E-mail address}:
\tt{a.taylor.maths@outlook.com}}
\end{flushleft}

\end{document}